\documentclass[reqno,12pt]{amsart}
\usepackage{amsmath,amssymb,amsthm,graphicx,mathrsfs,bbm,url}
\usepackage{amsthm}
\usepackage{wrapfig}
\usepackage{enumitem}
\usepackage{mathtools}
\usepackage[utf8]{inputenc}
 \usepackage[T1]{fontenc}
\usepackage[usenames,dvipsnames]{color}
\usepackage[colorlinks=true,linkcolor=Red,citecolor=Green]{hyperref}
\usepackage[super]{nth}
\usepackage[open, openlevel=2, depth=3, atend]{bookmark}
\hypersetup{pdfstartview=XYZ}
\usepackage[font=footnotesize]{caption}
\usepackage{a4wide}
\usepackage{epstopdf}

\theoremstyle{plain}
\newtheorem{theorem}{Theorem}[section]
\newtheorem{theoremexp}{Theorem}

\newtheorem*{theorem*}{Theorem}

\newtheorem{lemma}[theorem]{Lemma}
\newtheorem{proposition}[theorem]{Proposition}

\newtheorem{conjecture}[theorem]{Conjecture}

\theoremstyle{definition}

\theoremstyle{remark}
\newtheorem{remark}[theorem]{Remark}

\numberwithin{equation}{section}

\newcommand{\C}{\mathbb{C}}
\newcommand{\R}{\mathbb{R}}
\newcommand{\RR}{\mathbf{R}}

\newcommand{\M}{\mathcal{M}}
\newcommand{\E}{\mathcal{E}}
\newcommand{\N}{\mathbb{N}}

\newcommand{\V}{\mathbb{V}}
\newcommand{\X}{\mathbf{X}}

\newcommand{\HH}{\mathbb{H}}

\newcommand{\eps}{\varepsilon}

\newcommand{\mc}{\mathcal}

\newcommand{\dd}{\mathrm{d}}

\DeclareMathOperator{\Ell}{ell}
\DeclareMathOperator{\Tr}{Tr}
\DeclareMathOperator{\Op}{Op}
\DeclareMathOperator{\WF}{WF}

\newcommand{\be}{\begin{equation}}
\newcommand{\ee}{\end{equation}}

\title[Radial source estimates in H\"older-Zygmund spaces]{Radial source estimates in H\"older-Zygmund spaces for hyperbolic dynamics}

\author{Yannick Guedes Bonthonneau}

\address{LAGA - Institut Galilée, 99 avenue Jean Baptiste clément, 93430 Villetaneuse, France.}

\email{bonthonneau@math.univ-paris13.fr}

\author{Thibault Lefeuvre}

\address{Université de Paris and Sorbonne Université, CNRS, IMJ-PRG, F-75006 Paris, France.}

\email{tlefeuvre@imj-prg.fr}

\begin{document}

\begin{abstract}
We prove a radial source estimate in Hölder-Zygmund spaces for uniformly hyperbolic dynamics (also known as Anosov flows), in the spirit of Dyatlov-Zworski \cite{Dyatlov-Zworski-16}. The main consequence is a new linear stability estimate for the marked length spectrum rigidity conjecture, also known as the Burns-Katok \cite{Burns-Katok-85} conjecture. We show in particular that in any dimension $\geq 2$, in the space of negatively-curved metrics, $C^{3+\eps}$-close metrics with same marked length spectrum are isometric. This improves recent works of Guillarmou-Knieper and the second author \cite{Guillarmou-Lefeuvre-18,Guillarmou-Knieper-Lefeuvre-19}. As a byproduct, this approach also allows to retrieve various regularity statements known in hyperbolic dynamics and usually based on Journé's lemma: the smooth Liv\v{s}ic Theorem of de La Llave-Marco-Moriyón \cite{DeLaLlave-Marco-Moryon-86}, the smooth Liv\v{s}ic cocycle theorem of Niticā-Török \cite{Nitica-Torok-98} for general (finite-dimensional) Lie groups, the rigidity of the regularity of the foliation obtained by Hasselblatt \cite{Hasselblatt-92} and others. 
\end{abstract}

\maketitle

\section{Introduction}

\subsection{Motivation: the marked length spectrum conjecture}

Let $(M,g)$ be a smooth closed Riemannian manifold. We say that it is \emph{Anosov} if its geodesic flow $(\varphi_t)_{t \in \R}$ defined on its unit tangent bundle $\M := SM$ is Anosov, that is there exists a continuous splitting of the tangent space by
\[
T\M = \R X \oplus E_s \oplus E_u,
\]
and constants $C,\lambda>0$ such that
\begin{equation}
\label{equation:anosov}
\begin{array}{l}
\forall t \geq 0, \forall v \in E_s, \  \|\dd\varphi_t(v)\| \leq C\times e^{-t\lambda}\|v\|, \\
\forall t \leq 0, \forall v \in E_u, \  \|\dd\varphi_t(v)\| \leq C\times e^{-|t|\lambda}\|v\|,
\end{array}
\end{equation}
where $\| \cdot \| = g_\star(\cdot,\cdot)^{1/2}$ is an arbitrary smooth auxiliary metric on $\M$.

Let $\mc{C}$ be the set of free homotopy classes on $M$. It is well-known that, if $(M,g)$ is Anosov, there exists a unique closed geodesic $\gamma_g(c) \in c$ in each free homotopy class $c \in \mc{C}$. The \emph{marked length spectrum} of the Anosov manifold $(M,g)$ is the map
\begin{equation}
\label{equation:mls}
L_g : \mc{C} \rightarrow \R_+, ~~~ L_g(c) := \ell_g(\gamma_g(c)),
\end{equation}
where $\ell_g(\gamma)$ denotes the Riemannian length of the curve $\gamma$ computed with respect to the metric $g$.

Let $\mathrm{Met}_{\mathrm{An}}$ be the space of (smooth) Anosov metrics on $M$ and let $\mathrm{Diff}^0(M)$ be the group of smooth diffeomorphisms that are isotopic to the identity. It is clear that the map
\[
\mathrm{Met}_{\mathrm{An}} \ni g \mapsto L_{g}
\]
is invariant under the action (by pullback) of $\mathrm{Diff}^0(M)$, namely $L_g = L_{\phi^*g}$ whenever $\phi \in \mathrm{Diff}^0(M)$, so that $L_g = L_{[g]}$. An element $[g] \in \mathrm{Met}_{\mathrm{An}}/\mathrm{Diff}^0(M)$ is called \emph{an isometry class}. We are interested in the following conjecture, known as the \emph{Burns-Katok conjecture} \cite{Burns-Katok-85} or the \emph{marked length spectrum rigidity conjecture}\footnote{Originally, it was only formulated for negatively-curved manifolds.}:

\begin{conjecture}[Burns-Katok]
The map
\[
\mathrm{Met}_{\mathrm{An}}/\mathrm{Diff}^0(M) \ni [g] \mapsto L_{[g]}
\]
is injective.
\end{conjecture}

The conjecture was independently solved by Croke \cite{Croke-90} and Otal \cite{Otal-90} on negatively-curved surfaces. There are also a few other partial results \cite{Katok-88,Besson-Courtois-Gallot-95,Hamenstadt-99}. Recently, Guillarmou-Knieper and the second author \cite{Guillarmou-Lefeuvre-18,Guillarmou-Knieper-Lefeuvre-19} proved a \emph{local} version of the conjecture, namely assuming that $g$ and $g'$ are \emph{a priori} close enough in the $C^N$-topology (for some $N \gg 1$) with same marked length spectrum, they showed that the metrics are isometric. This result came with a \emph{stability estimate} but had two drawbacks: first of all, the integer $N$ was depending \emph{linearly} on the dimension $n$ of the manifold $M$; secondly, the stability estimate was non-linear. We address these two problems in our main Theorem \ref{theorem:mls}.

Before stating it, we need to introduce the notion of geodesic stretch. Given a fixed Anosov metric $g_0$ on $M$ (whose unit tangent bundle is denoted by $SM$) and another metric $g$, there is a function $a_g \in C^\nu(SM)$ (where $\nu > 0$ is a small exponent) with the property that:
\[
L_g(c) := \int_{\gamma_{g_0}(c)} a_g ~ \dd \gamma_{g_0}(c).
\]
We call it \emph{geodesic stretch}. In particular, $a_g-\mathbf{1}$ integrates to $0$ along all closed orbits of the geodesic flow of $g_0$ iff the two metrics have same marked length spectrum. The function $a_g$ is only well-defined up to a \emph{coboundary} i.e. a term of the form $X u$, where $X$ is the $g_0$-geodesic vector field and $u$ is a function on $SM$.

\begin{theorem}
\label{theorem:mls}
Let $(M,g_0)$ be a smooth Anosov manifold and further assume it is non-positively curved if $\mathrm{dim}(M) \geq 3$ or $g_0$ is generic. For any $\eps>0$, there exists $\nu, C > 0$ such that the following holds. For any metric $g$ such that $\|g-g_0\|_{C^{3+\eps}} < 1/C$, there exists a $C^{4+\eps}$-diffeomorphism $\phi$, isotopic to the identity, such that
\[
\| \phi^* g - g_0 \|_{C^{\nu-1}} \leq C \inf_{\substack{u \in C^\nu(SM), \\ Xu \in C^\nu(SM)}} \|a_g-\mathbf{1} + Xu\|_{C^{\nu}}.
\]
\end{theorem}
 
Theorem \ref{theorem:mls} shows that the metrics need only be $C^{3+\eps}$-close to obtain local rigidity of the marked length spectrum. The norm on the right-hand side in the quotient norm on the quotient space $C^\nu/D^{\nu}$ of functions up to coboundaries, see \S\ref{section:mls} for further details. As we shall see below, Theorem \ref{theorem:mls} is also based on the invertibility of a certain natural operator, the \emph{X-ray transform of symmetric $2$-tensors} for the metric $g_0$, denoted by $I_2^{g_0}$, and its injectivity is known on Anosov surfaces \cite{Paternain-Salo-Uhlmann-14-1, Guillarmou-17-1}, Anosov manifolds of dimension $n \geq 3$ with non-positive curvature \cite{Croke-Sharafutdinov-98}, and generic Anosov manifolds \cite{Cekic-Lefeuvre-21}, hence the restrictions in the theorem.

\subsection{Radial source estimates in Hölder-Zygmund spaces}

The technical result behind Theorem \ref{theorem:mls} is quite flexible and interesting in itself. We present it now, before turning to some remarks on regularity of cohomological equations. Let $(\varphi_t)_{t\in\R}$ be a smooth arbitrary Anosov flow in the sense of \eqref{equation:anosov} on a closed Riemannian manifold $(\M,g)$, with generating vector field $X$. Since our tools are that of microlocal analysis, we will work in $T^\ast\M$ rather than $T\M$, and we have to introduce the dual decomposition
\[
T^*\M = E_0^* \oplus E_s^* \oplus E_u^*,
\]
where $E_0^*(E_s\oplus E_u)=0, E_s^*(E_s\oplus \R X) = 0, E_u^*(E_u \oplus \R X)=0$. The flow $\varphi_t$ has a symplectic lift $\Phi_t$ to $T^\ast \M$ given by
\[
\Phi_t(x,\xi) = ( \varphi_t(x), d_x\varphi_t^{-\top} \xi)
\]
(here ${}^{-\top}$ stands for the inverse transpose). The same estimates as \eqref{equation:anosov} hold for $E_s$ replaced by $E_s^*$, $E_u$ replaced by $E_u^*$, $v \in T\M$ replaced by a covector $(x,\xi) \in T^*\M$ and $\dd \varphi_t(v) $ replaced by $\Phi_t(x,\xi)$.

We consider $\E\to\M$, a Hermitian vector bundle, and we assume given a derivation $\mathbf{X}$ acting on $C^\infty(\M,\E)$, lifting $X$ so that for $f\in C^\infty(\M)$ and $v\in C^\infty(\M,\E)$,
\begin{equation}
\label{equation:lie}
\X(f \cdot v) = Xf \cdot v + f \cdot \X v.
\end{equation}
The propagator $e^{-t \X}$ of the operator $\X$ is a pointwise map acting as:
\[
e^{-t\X} : \E_{\varphi_{-t}(x)} \rightarrow \E_x,
\]
for any $x \in \M$ and we denote by $M(t,x)$ its norm, namely:
\[
M(t,x) := \sup_{v \in \E_{\varphi_{-t}(x)}, \|v\|=1} \|e^{-t\X} v\|_x,
\]
where $\|\cdot\|$ stands for the norm in the fibers of $\E$. We introduce the following quantity
\begin{equation}
\label{equation:threshold}
\omega_+(\X) := \inf\left\{ \rho > 0 ~\middle|~ \sup_{x \in \M} \lim_{T\to + \infty} \frac{1}{T}\log \Big[M(T,x) \times \| \dd_x \varphi_{-T}|_{E^u} \|^\rho\Big] < 0 \right\},
\end{equation}
to which we will refer as the \emph{forward threshold} in the following. Observe that this quantity does not depend on a choice of metric on $\E$ nor on $g$. Reversing time and replacing $E^u$ by $E^s$, we obtain $\omega_-(\X)$ the backwards threshold. The key tool to proving Theorem \ref{theorem:mls} is the so-called \emph{propagation of singularities} toolbox and more precisely source estimates:

\begin{theorem}[Radial source estimate in Hölder-Zygmund regularity]
\label{theorem:source}
Let $A$ be a classical pseudodifferential operator of order $0$, microsupported in a small conical neighbourhood of $E^\ast_s$. There exists $B$, a classical pseudodifferential operator of order $0$, elliptic on the wavefront set of $A$, such that following holds. For $\rho>\omega_+(\X)$ and $N>0$, there exists $C>0$ such that for all $u\in C^\infty(\M,\E)$:
\begin{equation}\label{equation:source}
\| A u\|_{C^\rho_\ast} \leq C\left(\| B \X u \|_{C^\rho_\ast} + \| u \|_{C_*^{-N}}\right).
\end{equation}
More generally, if $u\in \mathcal{D}'(\M,\mathcal{E})$, and there exists $\rho_0 > \omega_+(\X)$ such that $Au \in C^{\rho_0}_\ast$ and $B \X u \in C^\rho_\ast$, then $A u \in C^{\rho}_\ast$ and \eqref{equation:source} holds.
\end{theorem}

Before presenting the relation with existing results, let us explain its meaning. Since the principal symbol $p$ of $\X$ is $1$-homogeneous, the flow $\Phi_t$ has a smooth extension to the radial compactification $\overline{T^\ast\M}$. In $\overline{T^\ast \M}$, $E^\ast_s \cap \partial T^\ast \M$ is invariant under $\Phi_t$. It is a repeller, in the sense that close trajectories get away exponentially fast. Since the works of Melrose \cite{Melrose-94}, one says that $E^\ast_s$ is a \emph{source}. The idea of propagation of singularities is that regularity propagates along flow lines of $\Phi_t$. If we have some regularity very close to $E^\ast_s$, then we can hope to propagate it to all points whose trajectory originates from close to $E^\ast_s$ is the past. The problem is then to obtain said regularity close to the source, and this is exactly what the source estimate is for.

\begin{center}
\begin{figure}[h!]

\includegraphics{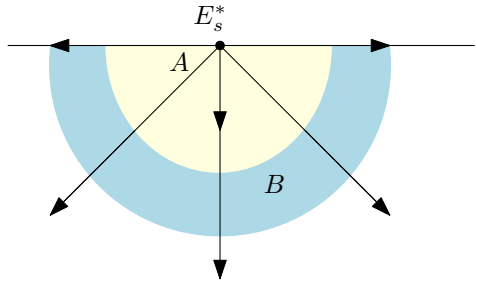}
\caption{A picture of the radial source estimate and the microsupport of $A$ and $B$.}
\label{figure:propagation}

\end{figure}
\end{center}

Melrose \cite{Melrose-94} was the first to introduce this kind of estimate. The purpose was to study Euclidean scattering for the Laplace operator. The idea was then expanded and generalized to several scattering problems \cite{Hassell-Melrose-Vasy-04,Dyatlov-12,Datchev-Dyatlov-13,Vasy-13,Hintz-Vasy-18,Dyatlov-Zworski-19-2,ColinDeVerdiere-18}. In these settings, sources are usually smooth manifolds but Dyatlov and Zworski \cite{Dyatlov-Zworski-16} observed that the same principle could be applied to Anosov flows (where the source, namely $E_s^*$, is only a Hölder-continuous distribution in $T^*\M$).

In \cite{Dyatlov-Zworski-16}, no explicit estimate was given on the threshold of regularity but later on, the authors gave a more precise statement in \cite[Appendix E]{Dyatlov-Zworski-19}. Although their development is quite recent, source estimates have now entered the standard toolbox of analysis of hyperbolic dynamics. Most relevant to our purposes, Guillarmou and de Poyferré \cite{Guillarmou-Poyferre-21} used paradifferential calculus to obtain a version of Dyatlov and Zworski's estimates for flows of finite regularity.

To the best of our knowledge, our result is the first source estimate involving Hölder-Zygmund instead of Sobolev norms. If we were giving the estimate for $L^2$-based spaces, the threshold $\omega_{\pm}(\X)$ would have to be replaced by another quantity, which is larger in general for the examples of interest to us. This motivated our research and is discussed with more details in Appendix \ref{appendix:b}. Additionally, Hölder-Zygmund regularity is a more natural choice when studying some dynamical problems (see \S\ref{section:cohomological} and \S\ref{section:mls} for applications). See also \cite{dlLLave-2001} for a discussion of Sobolev regularity vs H\"older regularity, and a similar estimate of regularity thresholds in the case of diffeomorphisms. 

Let us explain shortly why it is customary to use Sobolev norms when doing propagation of singularities for a certain (pseudo)differential operator $P$. There are two available schemes of proof for a source estimate. The first, technically simpler one, is using Egorov's Theorem. The second one relies on the so-called ``positive commutator argument'' of Hörmander. The latter relies on G\r{a}rding estimate, so that it can only be applied in a Hilbert space context. The first one uses crucially the propagator $e^{itP}$, which has to be bounded on the relevant spaces. It is well known that the wave or the Schrödinger propagators are not bounded on Hölder-Zygmund spaces. In fact, the only reasonable class of operators for which we have propagators bounded on $C^s_\ast$ spaces are vector fields. 

We are convinced that it would be possible to extend our result without much effort to the case of Besov spaces or other types of spaces based on interpolation and Littlewood-Paley constructions, still for vector fields. In a recent paper, Wang \cite{Wang-20} obtained a source estimate for more general type of operators in some $L^2$-based Besov spaces. His estimates involve some loss of regularity, but they apply to operators for which our proof could not work, because the propagator is not bounded in the relevant spaces. 

In a first draft of our article\footnote{Still available online \url{https://arxiv.org/abs/2011.06403}.} we had obtained a much weaker estimate on the value of the threshold. Given $A$ a pseudo-differential operator, the usual Egorov theorem gives a precise description of $A_t = e^{-tX}Ae^{tX}$. A weaker form of the statement just gives the wavefront set of $A_t$. It turns out that it is better for our purpose here to work only with the latter statement. This was suggested by the reading of \cite{Baladi-Tsujii-07}. This article of Baladi and Tsujii deals with Anosov diffeomorphisms, its methods were adapted to the case of flows by \cite{Adam-19} (see also \cite{Adam-Baladi-2021}).

\subsection{Regularity in hyperbolic dynamics}

Theorem \ref{theorem:source} has a straighforward consequence for regularity questions in hyperbolic dynamics. More precisely, it allows to give a systematic microlocal approach to studying regularity of solutions of \emph{cohomological equations} associated with the Anosov flow $(\varphi_t)_{t \in \R}$. The simplest such equation is $X u = f$, where $f$ is a given $C^\infty(\M)$ function. A satisfying criterion for existence of solutions is due to Liv\v{s}ic \cite{Livsic-72}, but it produces \emph{a priori} a \emph{Lipschitz-continuous} solution $u \in C^{\mathrm{Lip}}(\M)$. It was not until 1986 that de La Llave, Marco and Mory\'on \cite{DeLaLlave-Marco-Moryon-86} proved a bootstrap result, namely that if a Lipschitz-continuous solution $u$ exists, it is indeed $C^\infty(\M)$. This result was generalized in many directions, and used in several contexts. In this paper, we will consider cohomological equations on vector bundles and we assume that we have, as in the previous paragraph, a bundle $\E \to \M$ with Lie derivative $\X$.

From a microlocal point of view, the regularity of $u$ given that $\X u=f$ can be first understood with the principal symbol of $\mathbf{X}$. Equation \eqref{equation:lie} guarantees that the principal symbol of $-i\X$ is diagonal and given by $p(x,\xi) = \langle \xi,X(x)\rangle \mathbbm{1}_{\E}$. It follows that as a differential operator, $\X$ is elliptic in the direction $E^\ast_0$. In turn, elementary elliptic theory results imply that $\WF(u)\subset E^\ast_u\oplus E^\ast_s$. In simpler terms, $u$ is smooth along the orbits of the flow. 

On the characteristic set $\{p= 0\} = E^\ast_u \oplus E^\ast_s$, usual elliptic theory gives us no information, and we have to use more precise information on $\X$ than just its principal symbol to study the regularity of $u$. This extra-information is provided precisely by the radial source estimate of Theorem \ref{theorem:source} which tells us that microlocally near the source $E^\ast_s$ and the sink $E^\ast_u$, $u$ is smooth. By standard propagation of singularity, we can then conclude that $u$ is smooth microlocally everywhere on $\{p= 0\}$. We state this as the following result:

\begin{theorem}
\label{theorem:regularity}
Let $\rho> \max(0,\omega_+(\X),\omega_-(\X))$. Assume that $f\in C^{\infty}(\M,\E)$ and $u\in C^\rho(\M,\E)$ satisfy $\mathbf{X}u=f$. Then $u\in C^{\infty}(\M,\E)$.
\end{theorem}

Note that $\rho$ is always positive. The space $C^\rho(\M,\E)$ is the usual space of Hölder regularity for sections of $\E$. In the following, it will be more convenient for us to work with Hölder-Zygmund regularity $C^\rho_*(\M,\E)$ instead. The spaces $C^\rho(\M,\E)$ and $C^\rho_*(\M,\E)$ agree for $\rho > 0$ not equal to an integer but they are different for $\rho \in \N$, see \S\ref{section:pseudo} for further details. We also point out that one could extend to other regularities the statement: actually, it holds \emph{verbatim} with $C^\infty(\M,\E)$ being replaced by $C^s_*(\M,\E)$ for any $s \geq \rho$.

Theorem \ref{theorem:regularity} is not new. Indeed, it was proved in the case of the trivial bundle and $\X = X$ in \cite{DeLaLlave-Marco-Moryon-86}. Their proof uses the Journé Lemma \cite{Journe-86}, and has been adapted to deal with the general bundle case --- see for example \cite{Hasselblatt-92}. The use of microlocal estimates to deal with this kind of problems originates in \cite{Guillarmou-17-1}, where Guillarmou recovered the volume-preserving case of \cite{DeLaLlave-Marco-Moryon-86}. 

Theorem \ref{theorem:regularity} has strong consequences on rigidity questions in hyperbolic dynamical systems: it is the base of most of the standard regularity results \cite{Livsic-72, DeLaLlave-Marco-Moryon-86,Nitica-Torok-98} of the Liv\v{s}ic (cocycle) theory (see \S\ref{section:cohomological}, Theorems \ref{theorem:smoothness} and \ref{theorem:nt}) and the rigidity result of Hasselblatt \cite{Hasselblatt-92} relative to the smoothness of the Anosov foliation (see Theorem \ref{theorem:rigidity-foliation}).

\subsection{Organization of the paper} In Section \S\ref{section:pseudo}, we introduce Hölder-Zygmund spaces and show some elementary properties. In Section \S\ref{section:estimates}, we prove the source estimate of Theorem \ref{theorem:source}. We then turn to applications. In Section \S\ref{section:cohomological}, we apply Theorem \ref{theorem:source} to the study of (linear) cohomological equations in hyperbolic dynamical systems. We also discuss rigidity of the foliation of Anosov flows. In Section \S\ref{section:mls}, we apply the source estimate of Theorem \ref{theorem:source} to prove our main Theorem \ref{theorem:mls}.  \\

\noindent \textbf{Acknowledgement:} We thank Colin Guillarmou, Jian Wang, Maciej Zworski, Malo Jézéquel, Semyon Dyatlov and Viviane Baladi for fruitful discussions. We also thank Gabriel Paternain for suggesting the use of Ado's theorem in Theorem \ref{theorem:nt} in order to deal with general Lie groups. This project has received funding from the European Research Council (ERC) under the European Union’s Horizon 2020 research and innovation programme (grant agreement No. 725967).

\section{Pseudodifferential operators on Hölder-Zygmund spaces}

\label{section:pseudo}

The goal of this section is to introduce the scale of Hölder-Zygmund spaces $C^s_*$ (for $s \in \R$). We also provide several technical lemmata, describing the action of pseudo-differential operators on these spaces.

\subsection{Semiclassical analysis}

\label{ssection:analysis}

We briefly review some key features of semiclassical analysis on manifolds. For any $m \in \R$, $\rho,\delta \in [0,1]$, such that $1-\rho \leq \delta < \rho$, we define the Fréchet space $S^m_{\rho,\delta}(T^*\R^n)$ as the space of smooth functions $p$ on $T^*\R^n$ such that for all $\alpha,\beta \in \N^n$, there exists $C_{\alpha,\beta} > 0$ such that for all $(x,\xi) \in T^*\R^n$:
\begin{equation}
\label{equation:symboles}
|\partial^\alpha_\xi \partial^\beta_x p(x,\xi)| \leq C_{\alpha,\beta} \langle \xi \rangle^{m-\rho|\alpha|+\delta|\beta|}
\end{equation}
The standard space is $S^m_{1,0}(T^*\R^n)$ obtained for $\rho=1,\delta=0$ and we will denote it by $S^m(T^*\R^n)$ for the sake of simplicity.

On a manifold $\M$, we define the class $S^m_{\rho,\delta}(T^*\M)$ as the space of smooth functions $p \in C^\infty(T^*\M)$ satisfying \eqref{equation:symboles} in local charts. Since \eqref{equation:symboles} is locally invariant by the action of the group of diffeomorphisms, one can check that $S^m_{\rho,\delta}(T^*\M)$ is intrinsically defined on $\M$. We will also abuse notations, and allow symbols to depend on additionnal mute parameters (such as $h>0$), with the convention that \eqref{equation:symboles} is satisfied uniformly in those parameters.

On $\R^n$, we recall that the left quantization of a symbol $p \in S^m(T^*\R^n)$ is defined by:
\[
\Op(p)u(x) = \frac{1}{(2\pi)^n}\int_{\R^n_y} \int_{\R^n_\xi} e^{i \langle x-y,\xi\rangle} p\left(x,\xi\right) u(y) \dd y \dd \xi,
\]
and the semiclassical quantization is defined by
\[
\Op_h(p) := \Op(p(\cdot, h\cdot)).
\]
This also allows to define pseudodifferential operators on manifolds. We consider $(\kappa_i,U_i)$ a family of cutoff charts, namely a family of open sets such that $\M = \cup_{i=1}^N U_i$ and $\kappa_i : U_i \rightarrow \kappa_i(U_i) \subset \R^n$ is a diffeomorphism. We consider a partition of unity $\sum_{i=1}^N \Theta_i = \mathbf{1}$ subordinate to the cover $(U_i)_{i=1}^N$ and let $\Theta_i'$ be functions supported on $U_i$ such that $\Theta_i' \equiv 1$ on the support of $\Theta_i$. We let $\theta_i := \Theta_i \circ \kappa_i^{-1}$. Let $p \in S^m(T^*\M)$. We define:
\begin{equation}
\label{equation:quantization}
\Op_h(p) := \sum_{i=1}^N \kappa_i^*\theta_i \Op_h((\kappa_i^{-1})^*p) (\kappa_i^{-1})^*\Theta_i' ,
\end{equation}
where, for $(x,\xi) \in T^*\R^n$, $(\kappa_i^{-1})^*p(x,\xi) = p(\kappa_i^{-1}(x), \dd (\kappa_i^{-1})^{-\top}_x(\xi))$. The quantization procedure \eqref{equation:quantization} is highly non-canonical but one can define an intrinsic principal symbol map $\sigma_h : S^m(T^*\M) \rightarrow S^m(T^*\M)/hS^{m-1}(T^*M)$ such that $\Op_h(p) - \Op_h(\sigma_h(\Op_h(p))) \in h \Psi^{m-1}_h(\M)$. We refer to \cite[Appendix E]{Dyatlov-Zworski-19} for further details. The set of semiclassical pseudodifferential operators is then given by:
\[
\Psi^m_h := \left\{ \Op_h(p) + \mc{O}_{\Psi^{-\infty}_h}(h^\infty) ~|~ p \in S^m(T^*\M)\right\}.
\]
By $\mc{O}_{\Psi^{-\infty}_h}(h^\infty)$, we mean that this operator has smooth Schwartz kernel defined on $\M \times \M$ and that any of its derivatives (in local coordinates) is bounded by $C_N h^N$, for some constant $C_N > 0$, and for all $N \geq 0$.

In our arguments, we will have to deal with combinations of pseudo-differential operators, both classical and semi-classical, or semi-classical with different values of $h$. While this is not technically very difficult, it is a bit unusual, so we include a small discussion.

\begin{lemma}
\label{lemma:bound-composition}
Let $a, b \in S^0(T^*\M)$, so that $b$ is supported in $\{ |\xi|> 1\}$. For $0< h \leq h_0$, $\Op_{h_0}(a)\Op_h(b)$ is $h$-semi-classical, locally uniformly in $h_0$, in the sense that there exists a symbol $c \in S^0(T^*\M)$, supported in $\{ |\xi|>1\}$ such that
\begin{equation}
\label{equation:bound-composition}
\Op_{h_0}(a)\Op_h(b) = \Op_h(c) + \mathcal{O}(h^\infty),
\end{equation}
the remainder being smoothing with uniform estimates in $h_0$. The symbol $c$ may depend on $h$ and $h_0$. The same holds for the product $\Op_h(b)\Op_{h_0}(a)$. Additionally, $c$ can be chosen with same support as $b$, and if $a(\cdot, (h_0/h)\cdot)$ and $b$ have disjoint microsupport (uniformly in $h$ and $h_0$) then we can choose $c=0$. 
\end{lemma}

We point out here that this lemma will be applied at various places with different parameters. In particular, we will sometimes use the parameter $2^{-j}$ instead of $h_0$ or $h$.

\begin{proof}
We start by observing that if we add a $\mathcal{O}(h^\infty)$-smoothing term to $\Op_h(b)$, it does not change the result. This is also true if we add a $\mathcal{O}(h_0^\infty)$-smoothing term to $\Op_{h_0}(a)$. Indeed, if $R$ is such a term, we have for some uniform $C_{2N} > 0$ independent of $h,h_0$:
\[
\| R \Op_h(b)\|_{H^{-N}\to H^N} \leq C_{2N} \| \Op_h(b)\|_{H^{-N}\to H^{-3N}},
\]
where $H^s$ denotes the standard $h$-independent Sobolev space. Then, since $b$ is supported in $\{ |\xi|> 1 \}$, we can construct symbols $b_N$ of order $-2N$ for $N>0$ such that
\[
\Op_h(b) = (-h^2\Delta)^N\Op_h(b_N) + \mathcal{O}_{H^{-N}\to H^N}(h^N).
\]
Then, using that $\Op_h(b_N)$ is a $0$-th order classical pseudodifferential operator uniformly in $h$, it is uniformly bounded on $H^{-N}$ as $h\to 0$, so that 
\[
\| \Op_h(b)\|_{H^{-N}\to H^{-3N}} \leq C_N h^N. 
\]
From these observations, we deduce that we can work directly in $\R^n$. Indeed, when dealing with changes of chart and smooth cutoffs, the errors we have to deal with will be of the form we have just discussed. 

Another reduction that we can make is the following. We can write
\[
\Op_{h_0}(a) = \Op_{h_0}(a \chi + a(1-\chi)) = \Op_{h_0}(a_0+a_1),
\]
where $\chi$ is a cutoff (only depending on $\xi$) so that $a_0$ is supported in $|\xi| < h_0/2h$ and $a_1$ in $|\xi| > h_0/4h$. The part $\Op_{h_0}(a_1)$ will contribute by a $h$-semi-classical operator, since $\Op_{h_0}(a_1) = \Op_h (a_1(\cdot, h_0\cdot /h))$. This can be dealt with using usual results on compositions of such pseudors. We can thus assume that $a=a_0$ is supported for $|\xi|< h_0 /2h$, and it suffices to prove that the product is $\mc{O}(h^\infty)$-smoothing.

Using usual formulæ for the composition of pseudo-differential operators on $\R^n$ \cite[Theorem 4.14]{Zworski-12}, it suffices to prove that this is $\mathcal{O}(h^\infty \langle\xi\rangle^{-\infty})$ in symbol norm, uniformly in $h_0$:
\[
c (x, \xi) = \frac{1}{(2\pi )^n}\int_{\R^{2n}} e^{-i \langle z, \eta  \rangle } a\left(x, \frac{h_0}{h}\xi + h_0\eta\right) b(x+z, \xi) dz d\eta. 
\]
When estimating its symbol norms, from the formula, we see that the derivatives in $x$ do not pose a problem, and the derivatives in $\xi$ only do if they hit $a_0$. We are left with proving:
\[
I_N :=(h/h_0)^{-N}\left| \int e^{-i \langle z, \eta  \rangle } (\partial_\xi^N a_0)\left(x, \frac{h_0}{h}\xi + h_0\eta\right) b(x+z, \xi) dz d\eta \right| = \mathcal{O}(h^\infty \langle \xi\rangle^{-\infty}).
\]
Since $b$ is supported for $|\xi|>1$, and $a$ for $|\xi| < h_0/2h$, the domain of the integral does not encounter $\eta = 0$. This integral is thus non-stationary in the $z$ variable, so we integrate by part and find for $M>0$
\begin{align*}
I_N &= (h/h_0)^{-N} \left| \int e^{-i \langle z, \eta  \rangle } (\partial_\xi^N a_0)\left(x, \frac{h_0}{h}\xi + h_0\eta\right) \frac{\Delta_z^M b(x+z, \xi)}{|\eta|^{2M}} dz d\eta \right| \\
		&\leq C_M (h/h_0)^{-N} \int_{|\xi + h\eta|< 1/2} |\eta|^{-2M}  \langle \frac{h_0\xi}{h}+h_0\eta\rangle^{-N} d\eta. \\
		&\leq C_M (h/h_0)^{-N} h_0^{-n} \int_{|\eta|< 1/2} |\eta/h_0 - \xi/h|^{-2M}  \langle \eta \rangle^{-N} d\eta.\\
		&\leq C_M (h/h_0)^{-N} h_0^{-n} |\xi/h|^{-2M} \\
		&\leq C_M h_0^{N-n} h^{2M- N} \langle\xi\rangle^{-2M}.
\end{align*}
If $N \geq n$, taking $M$ arbitrarily large yields the claimed estimate. If $N < n$, one also has to use that $h_0 > h$. In both cases, the claim is proved.

\end{proof}

\subsection{Hölder-Zygmund spaces}
\label{ssection:hz-spaces}

The usual definition of the H\"older-Zygmund spaces is given by a description using the Littlewood-Paley decomposition that we recall. Let $\psi \in C^\infty(\R)$ be a smooth cutoff function such that $\psi \equiv 1$ on $[-1,1]$ and $\psi \equiv 0$ outside $[-2,2]$. We define $\varphi_0(\xi) = \psi(|\xi|)$ and $\varphi_j(\xi) = \psi(2^{-j}|\xi|) - \psi(2^{-j+1}|\xi|) = \varphi(2^{-j}|\xi|)$, where $\varphi(s) = \psi(s)-\psi(2s)$. It is then customary to set for $s \in \R$:
\[
\|u\|_{C^s_*(\R^n)} := \sup_{j \in \N} 2^{js} \|\Op(\varphi_j)u\|_{L^\infty}.
\]
In particular, one has the equality:
\[
u = \sum_{k=0}^{+\infty} \Op_1(\varphi_k) u,\ u\in \mathcal{S}'(\R^n).
\]
For the manifold $\M$, the $C^s_\ast(\M)$-norm is then defined via the use of charts as:
\begin{equation}
\label{equation:definition-co}
\|u\|_{C^s_*(\M)} := \sup_{i \in \left\{1,...,N\right\}} \sup_{j \in \N} 2^{js} \|\Op_1(\varphi_j) {\kappa_i^{-1}}^*(\Theta'_i u)\|_{L^\infty}.
\end{equation}

For our purposes, this is not really practical. Let us introduce a set of other norms. We consider $s\in \R$ and $h_0>0$. Then we let $\eps>0$ and take $\Phi \in C^\infty_{\mathrm{comp}}(T^*\M)$ a non-negative symbol supported in $\{(1+\eps)^{-1} <|\xi|< 1+\eps\}$, and equal to $1$ in $\{ (1+\eps/2)^{-1} < |\xi|< 1+\eps/2\}$. We also consider $\Psi \in C^\infty_{\mathrm{comp}}(T^*\M)$ supported in $\{ |\xi|< 3\}$, and equal to $1$ in $\{ |\xi|<2\}$. We let:
\[
\| u \|_{s,h_0,\Phi,\Psi} := \| \Op_{h_0}(\Psi) u \|_{L^\infty} + \sup_{0<h<h_0} h^{-s}\|\Op_h(\Phi)u\|_{L^\infty}.
\]
This is related to the $C^s_\ast$ norm:
\begin{lemma}\label{lemma:comparaison-norm}
The semi-norm $\| \cdot \|_{s,h_0,a,b}$ is a norm for $h_0$ small enough, equivalent to $\|\cdot\|_{C^s_\ast}$. More precisely, for $s\geq 0$, there exists $C > 0$ such that for all $h_0$ small enough:
\begin{align*}
\| u \|_{s,h_0,\Phi,\Psi} &\leq C \| u \|_{C^s_\ast} \\
\| u \|_{C^s_\ast} &\leq  C \left(h_0^{-s}\| \Op_{h_0}(\Psi) u \|_{L^\infty}  +  \sup_{0<h<h_0} h^{-s}\|\Op_h(\Phi)u\|_{L^\infty}\right).
\end{align*}
\end{lemma}

We will also need 
\begin{lemma}\label{lemma:comparaison-rho-rho'}
Let $\rho < \rho'$. Let $b \in S^0(T^*\M)$ be a $0$-th order symbol supported for $\left\{|\xi|>1/2\right\}$. Then, there exists $C > 0$ such that for all $h>0$ small enough, 
\[
\| \Op_h(b) u \|_{C^{\rho}_\ast} \leq Ch^{\rho'-\rho} \| \Op_h(b) u \|_{C^{\rho'}_\ast} + C h^N \|u\|_{C^{-N}}.
\]
On the other hand, if $b$ is supported in $\left\{|\xi|<2\right\}$, we have the converse estimate 
\[
\| \Op_h(b) u \|_{C^{\rho'}_\ast} \leq Ch^{\rho-\rho'} \| \Op_h(b) u \|_{C^{\rho}_\ast} + C h^N \|u\|_{C^{-N}}.
\]
\end{lemma}

In order to prove both lemmata, we will need this first lemma in $\R^n$:

\begin{lemma}\label{lemma:boundedness-phi}
For any compact annulus around $\left\{|\xi|=1\right\}$, there exists a constant $C > 0$ such that for all $a \in S^0(T^*\R^n)$ supported in that annulus, for $h\in(0,1]$: 
\[
\|\Op_h(a)\|_{\mc{L}(L^\infty,L^\infty)} \leq C  \sum_{|\alpha| \leq n +1}  \|\langle \xi \rangle^{|\alpha|} \partial^\alpha_\xi a\|_{L^\infty}.
\]
\end{lemma}

In particular, this lemma applies to products of the form $\Op_h(a)\Op(\varphi_j)$, or $\Op(a) \Op(\varphi_j)$, for $a\in S^0$, since those are microsupported in annuli, and $2^{-j}$-semi-classical according to Lemma \ref{lemma:bound-composition}. The proof is based on Schur's test. 

\begin{proof}
We may assume that $a$ is supported in $\{ 1/2 < |\xi|<2 \}$. We start with:
\[
\begin{split}
|\Op_h(a) u(x)| & = \dfrac{1}{(2\pi h)^n} \int_{\R^n_y} \int_{\R^n_\xi} e^{\frac{i}{h}\langle x-y , \xi \rangle} a(x,\xi) u(y) \dd y \dd \xi \\
&  \lesssim h^{-n} \int_{\R^n_z} \left| \int_{\R^n_\xi} e^{\frac{i}{h}\langle z, \xi \rangle} a(x,\xi)   \dd \xi \right| \dd z \|u\|_{L^\infty}  \\
	&\lesssim h^{-n} \|u\|_{L^\infty} \left( \int_{|z| \leq h} + \int_{|z| > h} \right)
\end{split}
\]
For the first integral, this is just a bound on the volume:
\[
\int_{|z| \leq h } \left| \int_{\R^n_\xi} e^{\frac{i}{h}\langle z, \xi \rangle} a(x,\xi) \dd \xi \right| \dd z \lesssim h^n    \|a\|_{L^\infty}.
\]
For the second integral, writing $D_{\xi_i} := i^{-1} \partial_{\xi_i}$ and $|D_\xi|^2 := \sum_{i=1}^n D_{\xi_i}^2$, and using that $h^2|D_\xi|^2(e^{\frac{i}{h} \langle z, \xi \rangle}) = |z|^2 e^{\frac{i}{h} \langle z, \xi \rangle}$, we obtain:
\[
\begin{split}
\int_{|z| > h  }  \left| \int_{\R^n_\xi} e^{\frac{i}{h}\langle z, \xi \rangle} a(x,\xi) \dd \xi \right|  \dd z  &= \int_{|z| > h }  \left| \int_{\R^n_\xi} |z|^{-2N} h^{2N} |D_\xi|^{2N}e^{\frac{i}{h}\langle z, \xi \rangle} a(x,\xi)  \dd \xi \right| \dd z\\
 	&	\lesssim  \int_{|z| > h }  |z|^{-2N} h^{2N} \int_{\R^n_\xi}  \left||D_\xi|^{2N} a(x,\xi)\right| \dd \xi\dd z
 \end{split}
 \]
 Note that:
 \[
 \left||D_\xi|^{2N}a (x,\xi)\right| \leq C_N \sum_{|\alpha| \leq 2N} \langle \xi \rangle^{-|\alpha|} \|\langle \xi \rangle^{|\alpha|} \partial^\alpha_\xi a\|_{L^\infty}  \mathbbm{1}_{\mathrm{supp}(a)}(\xi).
 \]
 Thus:
 \[
\begin{split}
 \int_{|z| > h  }  \left| \int_{\R^n_\xi} e^{\frac{i}{h}\langle z, \xi \rangle} a(x,\xi) \dd \xi \right| \dd z & \leq C_N   \sum_{|\alpha| \leq 2N}  \|\langle \xi \rangle^{|\alpha|} \partial^\alpha_\xi a\|_{L^\infty} \int_{|z| > h  }  |z|^{-2N} h^{2N} \int_{\R^n_\xi}   \mathbbm{1}_{\mathrm{supp}(a)}(\xi)\dd \xi\dd z\\
 & \leq C_N   \sum_{|\alpha| \leq 2N}  \|\langle \xi \rangle^{|\alpha|} \partial^\alpha_\xi a\|_{L^\infty}    \int_{|z| > h}  |z|^{-2N} h^{2N}  \dd z \\
 & \leq C_N   \sum_{|\alpha| \leq 2N}  \|\langle \xi \rangle^{|\alpha|} \partial^\alpha_\xi a\|_{L^\infty}  h^{2N} h^{n-2N} \\
 & \leq C_N   \sum_{|\alpha| \leq 2N}  \|\langle \xi \rangle^{|\alpha|} \partial^\alpha_\xi a\|_{L^\infty}  h^{n},
\end{split}
\]
and this holds as long as $N > n/2$. Note that the constants only depend on $N$ and the support of $a$ in the previous inequalities. This concludes the proof.
\end{proof}

We now turn to Lemma \ref{lemma:comparaison-norm}:

\begin{proof}[Proof of Lemma \ref{lemma:comparaison-norm}]
That this is a norm follows from the equivalence with the $C^s_\ast$ norm. We will need that $h_0>0$ is small enough to absorb $\mathcal{O}(h_0^\infty)$ remainders. Let us pick $u\in C^s_\ast$ and consider 
\[
\Op_h(\Phi) u= \sum_{j \geq 0} \Op_h(\Phi) \Op(\varphi_j) u. 
\]
In this sum, the terms with $2^{j+1}< h^{-1} /(1+\eps)$ are $\mathcal{O}(h^\infty)$ and smooth, according to Lemma \ref{lemma:bound-composition} (applied with parameters $h$ and $2^{-j}$ with $h < 2^{-j}$), which gives that
\[
\sum_{2^{j+1} < h^{-1}/(1+\eps)} \left\|\Op_h(\Phi) \Op(\varphi_j) u\right\|_{L^\infty} \leq C h^N \|u\|_{C^{-N}},
\]
On the other hand, the terms with $2^{j-1} > h^{-1}(1+\eps)$ are $\mathcal{O}( (2^{-j})^\infty )$ and smooth by the same Lemma \ref{lemma:bound-composition} (but applied with the reversed parameters) which gives that 
\[
\|\Op_h(\Phi) \Op(\varphi_j) u\|_{L^\infty} \leq C 2^{-jN} \|u\|_{C^{-N}}.
\]
From this we deduce that
\[
\|\Op_h(\Phi) u\|_{L^\infty} \leq C \sum_{h/(2+2\eps)< 2^{-j} < 2 h(1+\eps)}  \| \Op_h(\Phi)\Op(\varphi_j) u\|_{L^\infty} + C h^N \|u\|_{C^{-N}}.
\]
We can now apply Lemma \ref{lemma:boundedness-phi} to obtain for $h/(2+2\eps)< 2^{-j} < 2 h(1+\eps)$:
\[
\|\Op_h(\Phi)\Op(\varphi_j)u\|_{L^\infty} \leq C\|\Op(\varphi_j)u\|_{L^\infty} \leq C 2^{-js}\|u\|_{C^s_*} \leq C h^{-s}\|u\|_{C^s_*}.
\]
This implies:
\[
\sup_{0<h<h_0} h^{-s}\| \Op_h(\Phi)u \|_{L^\infty} \leq C \|u\|_{C^s_\ast},
\]
with $C$ depending on $\Phi$, and uniform as $h_0 \to 0$. We now consider the low frequencies. With the same arguments as before, we get
\[
\Op_{h_0}(\Psi)u = \sum_{2^{j-1} < 3h_0^{-1}} \Op_{h_0}(\Psi) \Op(\varphi_j) u + h_0^N\|u\|_{C^{-N}}. 
\]
We deduce that
\[
\| \Op_{h_0}(\Psi) u \|_{L^\infty} \leq C \| u\|_{C^s_\ast} \sum_{2^{j-1} < 3h_0^{-1}} 2^{-js} \leq C \| u \|_{C^s_\ast}.
\]
This proves the first inequality of Lemma \ref{lemma:comparaison-norm}.

Let us now consider the converse estimate. We need to control $\Op(\varphi_j)u$ by the $\Op_h(\Phi)u$. For $2^{j-1} > h_0^{-1}$, the same arguments as before will apply, and give a constant $C>0$ uniform in $h_0$ such that 
\[
\|\Op(\varphi_j)u \|_{L^\infty} \leq C \sup_{2^{-1-j}< h < 2^{1-j} } \|\Op_h(\Phi)u\|_{L^\infty} + C 2^{-jN} \|u \|_{C^{-N}}.
\]
It remains to control the terms with $2^{j-1} \leq h_0^{-1}$. For this we decompose $u= \Op_{h_0}(\Psi)u + \Op_{h_0}(1-\Psi)u$. Denoting the terms $u_0$ and $u_1$, we find that $\Op(\varphi_j)u_1$ is $h_0^N\|u\|_{H^{-N}}$ for $2^{j-1}\leq  h_0^{-1}$, so we can concentrate on $u_0$. Now, we can use the boundedness of $\Op(\varphi_j)$ on $L^\infty$ directly to get
\[
2^{js} \| \Op(\varphi_j)u_0 \|_{L^\infty}  \leq C 2^{js} \|u_0\|_{L^\infty}. 
\]
We deduce that
\[
\|u \|_{C^s_\ast} \leq C \|\Op_{h_0}(\Psi)u\|_{L^\infty} \underbrace{\max_{ h_0 2^{j-1}\leq 1} 2^{js}}_{\lesssim h_0^{-s}} + C \sup_{0<h<h_0} h^{-s} \| \Op_h(\Phi) u\|_{L^\infty} + C h_0^N \|u \|_{C^{-N}}.
\]
If $h_0$ is small enough, we can absorb the last term of the right-hand side in the left-hand side. This concludes the proof.
\end{proof}

We now prove Lemma \ref{lemma:comparaison-rho-rho'}:

\begin{proof}[Proof of Lemma \ref{lemma:comparaison-rho-rho'}]
Since the proofs are quite similar, let us concentrate on the case that $b$ is supported in $\left\{|\xi|<2\right\}$. Using Lemma \ref{lemma:bound-composition} (applied with parameters $2^{-j}$ and $h$, with $2^{-j} < h$), we have for $j$ sufficiently large enough (so that $2^{j+1} > 2/h$) that $\|\Op(\varphi_j)\Op_h(b)u\|_{L^\infty} \leq C_N 2^{-jN}\|u\|_{C^{-N}_*}$. Hence:
\[
2^{j \rho'} \|\Op(\varphi_j) \Op_h(b) u \|_{L^\infty} \leq C_N 2^{j(\rho'-N)} \|u\|_{C^{-N}_*} \leq C_N h^{N-\rho'} \|u\|_{C^{-N}_*}.
\]
For the remaining $j$'s (so that $2^{j+1} \leq 2/h$), we have:
\[
2^{j \rho'} \|\Op(\varphi_j) \Op_h(b) u \|_{L^\infty} \leq 2^{j(\rho'-\rho)} \underbrace{2^{j\rho}  \|\Op(\varphi_j) \Op_h(b) u \|_{L^\infty}}_{\lesssim \|u\|_{C^\rho_*}} \leq C h^{\rho-\rho'}  \|u\|_{C^\rho_*}.
\]
This proves the claim.
\end{proof}

It will be convenient to observe that:

\begin{lemma}\label{lemma:control-by-Cs}
Let $b \in S^0(T^*\M)$ be a $0$-th order symbol, supported in $\{|\xi|>1\}$. Then for $s>0$, there is a constant $C>0$ such that for all $h$ small enough, for all $u\in C^s_\ast$:
\[
\| \Op_h(b) u \|_{L^\infty} \leq C h^s \| u \|_{C^s_\ast}.
\]
\end{lemma}

\begin{proof}[Proof of Lemma \ref{lemma:control-by-Cs}]
The proof follows the lines of the start of the proof of Lemma \ref{lemma:comparaison-norm}. We leave the details to the reader. 
\end{proof}

We close this section with
\begin{theorem}[Calderon-Vaillancourt Theorem]
\label{theorem:cv}
For each $s\in \R$, there exists a constant $C > 0$ such that for all $a \in S^0(T^*\M)$, for all $h>0$:
\[
\|\Op_h(a)\|_{\mc{L}(C^s_*,C^s_*)} \leq C \sum_{|\alpha|\leq n+1} \|\langle \xi \rangle^\alpha \partial^\alpha_\xi a\|_{L^\infty}.
\]
\end{theorem}

Here, the term on the right-hand side is non-canonical and depends on the choice of cutoff charts to define $\Op_h$, namely 
\[
\|\langle \xi \rangle^\alpha \partial^\alpha_\xi a\|_{L^\infty} = \sup_{i\in \left\{1,...,N\right\}} \|\langle \xi \rangle^\alpha \partial^\alpha_\xi {\kappa_i^{-1}}^*a\|_{L^\infty(T^* \varphi_i(U_i))}.
\]
It is worthwhile to observe that the term on the right-hand side is invariant by scaling by $h$, namely the quantity
\[
\sup_{i\in \left\{1,...,N\right\}} \|\langle \xi \rangle^\alpha \partial^\alpha_\xi {\kappa_i^{-1}}^*a\|_{L^\infty(T^* \varphi_i(U_i))}
\]
is uniformly bounded (with respect to $h > 0$) if one replaces $a$ by $a(\cdot,h\cdot)$.

\begin{proof}
By construction of the semiclassical quantization \eqref{equation:quantization}, this boils down to proving the statement in $\R^n$. The proof is almost the same for different values of $s$, so we deal with the case $s=0$. Let us compute:
\[
\begin{split}
\|\Op_h(a)u\|_{C^0_*} & = \sup_{j \in \N} \|\Op(\varphi_j) \Op_h(a)u\|_{L^\infty} \\
& \leq  \sup_{j \in \N} \sum_{k=0}^{+\infty} \|\Op(\varphi_j) \Op_h(a)\Op(\varphi_k)(u_{k-1}+u_k+u_{k+1})\|_{L^\infty} 
\end{split}
\]
where $u_k := \Op(\varphi_k)u$. We split the sum above in $|k-j| \leq 1$ and $|k-j| > 1$. By Lemma \ref{lemma:boundedness-phi}, we have:
\[
\sum_{|k-j| \leq 1}^{+\infty} \|\Op(\varphi_j) \Op_h(a)\Op(\varphi_k)(u_{k-1}+u_k+u_{k+1})\|_{L^\infty} \leq C \sum_{|\alpha|\leq n+1} \|\langle \xi \rangle^\alpha \partial^\alpha_\xi a\|_{L^\infty} \|u\|_{C^0_*}.
\]
Then, we claim that for all $N \geq 0$, there exists a constant $C_N > 0$ (independent of $j$ and $k$ but depending on $a$) such that for all $|j-k| \geq 1$:
\begin{equation}
\label{equation:decroissance-off-diagonal}
\|\Op(\varphi_j) \Op_h(a)\Op(\varphi_k)\|_{\mc{L}(L^\infty,L^\infty)} \leq C_N 2^{-N\max(j,k)}.
\end{equation}
This follows actually from Lemma \ref{lemma:bound-composition} (applied with parameters $h$ and $2^{-\max(j,k)}$). As a consequence, for any $N > 0$:
\[
\begin{split}
\sum_{|k-j| > 1}^{+\infty} \|\Op(\varphi_j) \Op_h(a)\Op(\varphi_k)(u_{k-1}+u_k+u_{k+1})\|_{L^\infty} & \leq \sum_{|k-j| > 1}^{+\infty} C_N 2^{-N\max(j,k)} \|u\|_{C^0_*} \\
& \leq C_N  \|u\|_{C^0_*}
\end{split}
\]
\end{proof}

\section{Radial estimates}
\label{section:estimates}

This section is devoted to the proof of Theorem \ref{theorem:source}. After some remarks, the proof will be divided into a semi-classical estimate, the core of the proof, and a regularity bootstrap to conclude. 

\subsection{Standard propagation of singularities}

Radial estimates are a refinement of the more elementary \emph{propagation of singularities}. It is customary to present such estimates in the framework of Sobolev spaces. The reasons for this are twofold. First, the scheme of proof proposed by Duistermaat-Hörmander \cite{Duistermaat-Hormander-72}, using a positive commutator and G\r{a}rding estimate can only work on spaces based on $L^2$. The other available scheme of proofs relies on using the propagator. It is more versatile, provided the propagator is bounded. As we recalled before, the only interesting class of operators whose propagators are known to be bounded on spaces more general than $L^2$ are exactly vector fields and potentials. 

Propagation of singularities for flows is actually quite simple. We give a proof as a warmup for the following discussion. If $X$ is a vector field, $(\varphi_t)_{t \in \R}$ the associated flow, we denote by $(\Phi_t)_{t \in \R}$ the symplectic lift of the flow to $T^\ast \M$. Recall that $\Phi_t(\cdot) = (\varphi_t(\cdot), \dd \varphi_t^{-\top}(\cdot))$, where ${}^{-\top}$ stands for the inverse transpose. We let $\E \rightarrow \M$ be a Hermitian vector bundle and $\X$ be a Lie derivative on $\E$ as in \eqref{equation:lie}. The following standard propagation of singularities holds for any Lie derivative $\X$, independently of the hyperbolic nature (or not) of $X$.

\begin{proposition}\label{prop:usual-propagation-Cs}
Let $A,B,D\in \Psi^0(\M,\E)$ with diagonal principal symbol and $T>0$ be such that $D$ is elliptic on $\Phi_T(\WF(A))$ and $B$ is elliptic on $\Phi_t (\WF(A))$ for all $t\in [0,T]$. Then for each $s\in \R$, there exists a constant $C>0$ such that
\[
\| A u \|_{C_*^s} \leq C \left(\| B \X u \|_{C_*^s} + \| D u \|_{C_*^s} + \|u\|_{C_*^{-N}}\right).
\]
\end{proposition}

\begin{proof}
We observe that
\[
A u = \int_0^T A e^{-t\X}dt \X u  + A e^{-T\X}u. 
\]
We can use the Egorov Lemma to observe that $A e^{-t\X}  = e^{-t\X} A_t$, where $A_t\in \Psi^0(\M,\E)$ satisfies
\[
\WF(A_t)\subset \Phi_t (\WF(A)).
\]
Hence, using that $e^{-t\X}$ is bounded on $C^s_\ast$ for $s\in\R$, we get:
\[
\|Au\|_{C^s_*} \leq C_T \left(\int_0^T \|A_t \X u\|_{C^s_*} \dd t + \|A_T u\|_{C^s_*}\right),
\]
for some constant $C_T > 0$. Then, we use an elliptic parametrix to control $A_t$ by $B$ (that is $A_t = A'_t B + \mc{O}_{\Psi^{-\infty}}(1)$ for some $A'_t$ of order $0$), and $A_T$ by $D$. Finally, we use the boundedness of $\Psi^0(\M,\E)$ pseudors (by Theorem \ref{theorem:cv}) to conclude.
\end{proof}

\subsection{A semi-classical version of the source estimate}

We turn now to the heart of the article and assume that $X$ is Anosov in the sense of \eqref{equation:anosov}. We have the following semiclassical analogue of Theorem \ref{theorem:source}:

\begin{theorem}
\label{theorem:sc-source}
Let $A_{h_0}$ be a $h_0$-semiclassical pseudodifferential operator of order $0$, microsupported in a small enough neighbourhood of $E^\ast_s \cap \partial T^*\M$. There exists $B_{h_0}$, a $h_0$-semiclassical pseudodifferential operator of order $0$, elliptic on the wavefront set of $A_{h_0}$, such that following holds. For $\rho>\omega(\X)$, and $N>0$, there exist $C>0$ such that for $u\in C^\infty(\M,\E)$, and $h_0$ small enough:
\begin{equation}
\label{equation:sc-source}
\| A_{h_0} u\|_{C^\rho_\ast} \leq C \| B_{h_0} \X u \|_{C^\rho_\ast} + Ch_0^N \| u \|_{C_*^{-N}}.
\end{equation}
Moreover, if $u$ is a distribution such that $A_{h_0} u \in C^\rho_\ast$, and $B_{h_0} \X u \in C^\rho_\ast$, then the previous equality still holds.
\end{theorem}

The fact that \eqref{equation:sc-source} extends to distributions is a straightforward consequence of the density of $C^\infty$ in distributions. The statement of Theorem \ref{theorem:source} is stronger as it shows that if the right-hand side exists, then so does the left-hand side. We will not prove this ``bootstrap" statement in the semiclassical setting but only in the classical setting, and we leave it as an exercise for the reader. This semiclassical version of the source estimate will eventually allow us to prove the full classical statement in the introduction, namely Theorem \ref{theorem:source}.

Our argument could probably be adapted to deal with the case of a negative threshold. However, this would complicate a bit the proof, and in practical situations we have encountered, the threshold is always non-negative. We also believe that this could be generalized to the case where the bundle $\E$ is infinite-dimensional: this could lead to stronger regularity statements in the Liv\v{s}ic cocycle theory, as one could take Lie groups $G$ with infinite-dimensional Lie algebra such as $\mathrm{Diff}(M)$ for instance.

\begin{proof}[Proof of Theorem \ref{theorem:sc-source}] We prove Theorem \ref{theorem:sc-source} in two steps: 1) for smooth functions with a certain threshold condition; 2) we then relate this threshold condition with the $\omega(\X)$ defined in \eqref{equation:threshold}. Also note the following straightforward reduction: if the estimate holds for elliptic pseudodifferential operators $A_{h_0}$'s, then it surely holds for all $A_{h_0}$'s. We can therefore assume $A_{h_0}$ is elliptic and microlocally equal to the identity in a neighorhood of $E_s^* \cap \partial T^*\M$.\\

\emph{Step 1:} We start by establishing the estimate for $u\in C^\infty(\M,\E)$. The fact that $u$ is a section of a bundle will only appear at a later stage (when estimating $L^\infty$-norms of some pseudodifferential operators). As a consequence, we can forget about the twist by $\E$ for the moment. All the operators are assumed to have diagonal principal symbol.

Up to loosing a $h_0^\infty$-smoothing remainder, we can always assume that $A_{h_0}:=\Op_{h_0}(\tilde{a})$, where 
\[
\tilde{a}(x,\xi) = a(x,\xi) \chi( |\xi|),
\]
where $\chi\in C^\infty(\R)$ is equal to $1$ on $[5,+\infty[$, and vanishes on $]-\infty, 4]$. Using some elliptic estimate, we can also assume that $a(x,\xi)$ is $0$-homogeneous. We denote by $\mathcal{C}_1 \subset \mathcal{C}_0$ the two conic neighbourhoods of $\overline{E^\ast_s}\cap \partial T^\ast \M$ given by $\mathcal{C}_0:=\{ a \neq 0,\ |\xi|>4\}$ and $\mathcal{C}_1:=\{ a = 1,|\xi|>10\}$. 

For the pseudor $B$, we proceed as follows. We pick $b'$ equal to $1$ on 
\[
\mathcal{C}'_0 := \cup_{t>0} \Phi_{-t}(\mathcal{C}_0),
\]
and supported in a conical neighbourhood thereof. (Here $\Phi_t := (\varphi_t, \dd \varphi_t^{-\top}(\cdot))$ is the symplectic lift of $\varphi_t$.) Next, we choose $b$ also with conical support, and equal to $1$ on the support of $b'$. Then, we define $B_{h_0}= \Op_{h_0}(b)$, and $B'_h= \Op_h(b')$. (Since $E^\ast_s$ is a source, this is a small conical neighbourhood of $\overline{E^\ast_s}\cap \partial T^\ast \M$). 

Next, with $\Phi$ being the function that appeared in the dyadic decomposition, we observe that for $0<h<h_0$,
\begin{equation}
\label{equation:simplification}
\Op_h(\Phi) A_{h_0} = \underbrace{\Op_h( a' )}_{:=A'_h} + \mc{O}_{\Psi^{-\infty}_h}(h^\infty),
\end{equation}
with $a'$ supported in $\mathcal{C}_0\cap \{ 1/2 < |\xi|<2\}$, by Lemma \ref{lemma:bound-composition}. We denote $A'_h := \Op_h(a')$. As $A_{h_0}$ has no wavefront set for $\xi$ near $0$, we know by Lemma \ref{lemma:comparaison-norm} that there exists $C > 0$ (independent of $h_0$) such that:
\begin{equation*}
\| A_{h_0} u \|_{C^\rho_\ast} \leq  C \left(h_0^{-\rho} \|\Op_{h_0}(\Psi)A_{h_0} u \|_{L^\infty} +  \sup_{0<h<h_0} h^{-\rho} \| \Op_h(\Phi) A_{h_0} u \|_{L^\infty}\right) 
\end{equation*}
We use the fact that $\Op_{h_0}(\Psi)$ and $A_{h_0}$ have distinct microsupport (and Lemma \ref{lemma:bound-composition}) to deduce that for any $N>0$, there exists $C>0$ such that
\begin{equation}\label{eq:reducing-to-semi-classical-estimate}
\| A_{h_0} u \|_{C^\rho_\ast} \leq C \left( h_0^N \| u \|_{C^{-N}} + \sup_{0<h<h_0} h^{-\rho} \| \Op_h(\Phi) A_{h_0} u \|_{L^\infty}\right).
\end{equation}
Taking $h_0>0$ small enough and using \eqref{equation:simplification}, we can replace $\Op_h(\Phi) A_{h_0}$ by $A'_h$ in the right-hand side of \eqref{eq:reducing-to-semi-classical-estimate} (the error term is $\mathcal{O}(h_0^\infty)$ and can be absorbed in the $\|u\|_{C^{-N}}$ term). We have reduced the problem to estimating $\| A'_h u \|_{L^\infty}$. 

After this reduction, we consider the formula for $T>0$
\begin{equation}
\label{equation:ipp}
A_h' = A'_he^{-T\X} + \int_0^{T} A'_h e^{-t\X} \X.
\end{equation}
We deal first the ``transport part'' of the right-hand side. For $0 \leq t \leq T$:
\[
A'_h e^{-t \X} = A'_h e^{-t \X} B_h' + \mc{O}_{\Psi^{-\infty}_h}((he^{\lambda t})^\infty).
\]
This is a mere consequence of Egorov's theorem (more precisely: of the propagation of singularities by the operator $e^{-t \X}$), and the choices we have made on the support of $b$.
We insist on the fact that the constant $T > 0$ will be chosen large enough in the end and independent of $h$ (as a consequence, we do not even need to propagate up to Ehrenfest time); then $h$ will be chosen small enough. By $\mc{O}_{\Psi^{-\infty}_h}((he^{\lambda t})^\infty)$, we mean that this operator is smoothing (its kernel $K_h(t)$ is a smooth function on $\M \times \M$) and that for any $N \geq 0$, one has $\|K_h(t)\|_{C^N(\M \times \M)} \leq h^N e^{\lambda N t}$, for all $h > 0$ small enough and $0 \leq t \leq T$. 

Fixing an arbitrary $N \gg 1$, we deduce that for $T>0$, using \eqref{equation:ipp}:
\[
\begin{split}
\left\| A'_h \int_0^T e^{-t \X} \X u ~~\dd t \right\|_{L^\infty} & \lesssim \| B_h'\X u\|_{L^\infty} \int_0^T \| e^{-t \X}\|_{L^\infty \rightarrow L^\infty} dt + (he^{\lambda T})^N \|u\|_{C^{-N}} \\
& \lesssim e^{\lambda T} \| B_h'\X u\|_{L^\infty} + (he^{\lambda T})^N \|u\|_{C^{-N}}.
\end{split}
\]
(Let us also insist on the fact that $\lesssim$ refers to universal constants which are neither dependent on $h$ nor on $T$. Also, $\lambda$ might be different from one line to another.) Since $\rho>0$ and $B'_h$ is microsupported in a region where $B_{h_0}$ is microlocally the identity, we can use Lemma \ref{lemma:bound-composition} to see that
\[
B'_h = B'_h B_{h_0} + \mathcal{O}_{\Psi^{-\infty}_h}(h^\infty),
\]
and then apply Lemma \ref{lemma:control-by-Cs} to $B'_h$, to find
\[
\|B_h' \X u \|_{L^\infty} \lesssim h^{\rho}\| B_{h_0}\X u\|_{C^\rho_\ast} + h^N\|u\|_{C_*^{-N}}.
\]
From the arguments above, and using \eqref{equation:ipp}, we obtain:
\begin{equation}
\label{equation:partiel}
\| A'_h u \|_{L^\infty} \lesssim e^{\lambda T} h^\rho \| B_{h_0}\X u \|_{C^\rho_\ast} + \| A'_h e^{-T\X} u\|_{L^\infty} + (he^{\lambda T})^N\|u\|_{C_*^{-N}}.
\end{equation}
The proof will thus be complete if we find $T>0$ large enough so that 
\begin{equation}
\label{equation:todo}
\sup_{0<h<h_0} h^{-\rho} \| A'_h e^{-T\X} u\|_{L^\infty} \leq \frac{1}{1994} \| A_{h_0} u\|_{C^\rho_\ast} + C_T h_0^N \|u\|_{C_*^{-N}},
\end{equation}
for some constant $C_T > 0$. Indeed, if this is the case, then dividing by $h^\rho$ in \eqref{equation:partiel} and taking the $\sup_{0 < h < h_0}$, we will obtain using \eqref{eq:reducing-to-semi-classical-estimate} that:
\[
\|A_{h_0} u\|_{C^\rho_*} \leq C_T \|B \X u \|_{C^\rho_\ast} + \frac{1}{1994} \| A_{h_0} u\|_{C^\rho_\ast}  + C_T h_0^N \|u\|_{C_*^{-N}},
\]
where $C_T > 0$ might be different from one line to another. Then, the term $\frac{1}{1994} \| A u\|_{C^\rho_\ast}$ can be put in the left-hand side and this completes the proof.

Let us turn to the proof of \eqref{equation:todo}. Since $E^\ast_s$ is a source, there exists $T_0>0$ such that for $T>T_0$, $\overline{\Phi_{-T}(\mathcal{C}_0)} \subset (\mathcal{C}_1)^{\circ}$ and thus in particular:
\[
A'_h e^{-T\X} = A'_h e^{-T \X} A_{h_0} + \mc{O}_{\Psi^{-\infty}_h}((he^{\lambda T})^\infty).
\]
We will need to improve this formula. By Egorov's Lemma, we know that (using the notation ${}^{-\top}$ for the inverse transpose):
\begin{align*}
\WF_h( e^{T\X} A'_h e^{-T\X} ) &\subset \Phi_{-T}(\WF_h(A'_h) )\\
						&\subset \left\{(x,\xi)\ |\ (y,\eta):=(\varphi_T(x), \dd _x\varphi_T^{-\top}\xi) \in \mathcal{C}_0(\varphi_T(x)),\ 1/2 < |\eta| < 2 \right\} \\
						&\subset \left\{(x,\xi)\ |\ (y,\eta):=(\varphi_T(x), \dd _x\varphi_T^{-\top}\xi) \in \mathcal{C}_0(\varphi_T(x)),\ 1/2 < |\eta| \right\} \\
						&\subset \left\{(x,\xi)\in \mathcal{C}_1 ~|~ \ 2|\xi| > \inf_{\eta\in \mathcal{C}_0(\varphi_T(x))} \frac{ |\dd_x\varphi_T^{\top}\eta|}{|\eta|} \right\},
\end{align*}
as long as $T$ is large enough.
We define 
\[
\Lambda_{\mathcal{C}_0}(x,t) :=  \sup_{\eta\in \mathcal{C}_0(\varphi_T(x))} \frac{|\eta|}{ |(\dd_x \varphi_T)^{\top} \eta|}. 
\]
Observe that equivalently
\[
\Lambda^{-1}_{\mathcal{C}_0}(x,t)= \inf_{\eta\in \mathcal{C}_0(\varphi_T(x))} \frac{ |\dd_x\varphi_T^{\top}\eta|}{|\eta|},
\]
and this increases exponentially in time. Let
\[
p_T(x,\xi) := a(x,\xi)\chi( 2|\xi| \Lambda_{\mathcal{C}_0}(x,t)), 
\]
where $\chi \in C^\infty(\R)$ is a cutoff function such that $\chi \equiv 1$ for $|x| > 1$ and $\chi \equiv 0$ for $|x| < 1/2$. By construction, this is equal to $1$ for $(x,\xi) \in T^*\M$ such that $(x,\xi) \in \mathcal{C}_1$ and $2 |\xi|> \Lambda^{-1}_{\mathcal{C}_0}(x,t)$ (and supported in $(x,\xi) \in \mathcal{C}_1$ and $4 |\xi|> \Lambda^{-1}_{\mathcal{C}_0}(x,t)$).
Note that, by construction, $p_T$ is microlocally the identity on the wavefront set of $e^{TX} A'_h e^{-TX}$.

Moreover, from the homogeneous structure of the symbol $p_T$, we observe in any local patch of coordinates $U \subset \R^n$, for all $\alpha\in \N^n$, for all $(x,\xi) \in T^*U$ and $T \geq 0$,
\[
\sup_{(x,\xi) \in T^*U} |\langle\xi\rangle^{|\alpha|} \partial_\xi^\alpha p_T(x,\xi)| < C_\alpha,
\]
where $C_\alpha$ is independent of $T$. (On the other hand, the derivatives in $x$ increase exponentially in size with time). Using Egorov's Lemma again, with $P_h^T := \Op_h(p_T)$, we find that:
\begin{equation}
\label{equation:this-is-life}
A'_h e^{-T \X} = A'_h e^{-T \X} P_h^T A_{h_0} + \mc{O}_{T, \Psi^{-\infty}_h}(h^\infty),
\end{equation}
where the remainder depends on $T$. Let us introduce a bit of notation. We will use crucially that $e^{-T\X}$ acts in a local fashion. More precisely, the propagator is a pointwise (in $x \in \M$) linear map 
\[
e^{-T \X}(x)\in \mathrm{Hom}(\E_{\varphi_{-t}(x)}, \E_x),
\]
and thus we can consider its norm
\[
M(T,x) := \sup_{u \in \E_{\varphi_{-t}(x)}, \|u\|=1} \|e^{-T \X}u\|_{\E_x}.
\] 
We now claim that the following holds:
\begin{lemma}\label{lemma:exponential-decay-estimate}
Let $\rho> 0$, and assume that:
\[
K:=\lim_{T\to + \infty} \frac{1}{T} \sup_{x \in \M} \log \left( M(T,x) \times \Lambda_{\mathcal{C}_0}(\phi_{-T}(x),T)^\rho\right) < 0.
\]
Then for any $\eps>0$, there exists $T_0\geq0$ such that for all $T> T_0$, 
\[
\| e^{-T\X} P_h^T u \|_{L^\infty} < \eps h^\rho\| u\|_{C^\rho_*} 
\]
\end{lemma}

Let us assume that Lemma \ref{lemma:exponential-decay-estimate} holds. Then using \eqref{equation:this-is-life} in the first line, and in the second line the fact that $A'_h : L^\infty \rightarrow L^\infty$ is bounded independently of $h > 0$, according to Lemma \ref{lemma:boundedness-phi} and \eqref{equation:simplification}, we get:
\[
\begin{split}
\|A'_h e^{-T \X} u\|_{L^\infty} &\leq \|A'_h e^{-T \X} P_h^T A_{h_0} u\|_{L^\infty} + h^NC_T \|u\|_{C^{-N}_*} \\
& \leq C \|e^{-T \X} P_h^T A_{h_0} u\|_{L^\infty} + h^N C_T \|u\|_{C^{-N}_*}  \\
& \leq C \eps h^\rho \|A_{h_0}u\|_{C^\rho_*} + h^N C_T \|u\|_{C^{-N}_*}
\end{split}
\]
Taking $\eps = 1/(C\times1994)$ and $T$ larger than the corresponding $T_0 \geq 0$, we obtain the desired estimate \eqref{equation:todo}.

\begin{proof}[Proof of Lemma \ref{lemma:exponential-decay-estimate}]
Since the flow acts pointwise, we will concentrate on finding pointwise estimates for $P^T_h u (x)$. By construction, $P^T_h = \Op_h(p_T)$ and the quantization in defined in \eqref{equation:quantization}. Hence, in order to estimate $|P^T_h u(x)|$ up to a fixed multiplicative constant, it suffices to find pointwise estimates on
\[
P_U^T u := \Op_h( \kappa_* p_T) \left( \kappa_*(\Theta u)\right),
\]
where $U$ is one of the $U_i$'s in \eqref{equation:quantization} (and $\Theta = \Theta_i, \kappa=\kappa_i$). Since we are working with the left quantization, this is particularly simple. Indeed, for $f$ supported in the interior of $U' := \kappa(U)$
\[
P_U^T f(x) = \int_{\R^n} e^{i \langle x, \xi\rangle} \left[\kappa_\ast p_T \right](x,h\xi)\hat{ f }(\xi) \dd \xi.
\]
By construction, $(\kappa)_\ast p_T$ is only supported for $|\xi| > 1/(4\Lambda_{\mathcal{C}_0}(x,t))$, so that if $f = \sum_{j\geq 0} f_j$ is the Littlewood-Paley decomposition of \S\ref{ssection:hz-spaces}, where each $f_j := \Op(\varphi_j)f$ has its Fourier transform supported in $\left\{2^{j-1} \leq |\xi| \leq 2^{j+1}\right\}$, we have:
\[
P_U^T f = \sum_{j + 1\geq  |\log 4 h \Lambda_{\mathcal{C}}(x,t)| / \log 2 } P^T_U f_j.
\]
However, by Lemma \ref{lemma:boundedness-phi}
\[
| P_U^T f_j(x) |  \lesssim \left(\sup_{|\alpha| \leq n+1 } \sup_{(x,\xi) \in T^*U'} | \langle\xi\rangle^{|\alpha|}\partial_\xi^\alpha p_T(x, \xi)|\right) (\| f_{j-1} \|_{L^\infty} + \| f_j \|_{L^\infty} + \| f_{j+1} \|_{L^\infty}).
\]
Since we have chosen $p_T$ homogeneous of degree $0$, these derivatives are actually controlled independently of $T$, namely we obtain for some constant $C >0$ independent of $T$:
\[
|P_U^T f(x) | \lesssim \sum_{j + 2\geq  |\log 4 h \Lambda_{\mathcal{C}}(x,t)| / \log 2 } \|f_j\|_{L^\infty}.
\]
By definition of the $C^\rho_\ast$ norm (see \eqref{equation:definition-co}), it follows that :
\[
|P_U^T f(x) | \lesssim \left| h \Lambda_{\mathcal{C}_0}(x,t) \right|^{\rho} \left(\sum_{j \geq 0} 2^{-j\rho}\right) \|f\|_{C^\rho_\ast} \lesssim h^\rho |\Lambda_{\mathcal{C}_0}(x,t)|^{\rho} \|f\|_{C^\rho_\ast}.
\]
(We see that the assumption $\rho>0$ is necessary here.) Gathering our estimates, we find:
\[
\left\| e^{-T \X} P^T_h u \right\|_{L^\infty} \lesssim h^\rho \| u \|_{C^\rho_\ast} \sup_{x \in \M} \left( M(T,x) \times \Lambda_{\mathcal{C}_0}(\varphi_{-T}(x),T)^\rho \right).
\]
It follows from the assumption on $\rho$ that for any $\delta>0$, there exists a constant $C>0$ such that for $T>0$
\[
 \sup_{x \in \M}  M(T,x) \times \Lambda_{\mathcal{C}_0}(\varphi_{-T}(x),T)^\rho  \leq e^{C-(K-\delta) T}.
\]
Taking $T_0>0$ large enough and $T>T_0$ we find that the desired estimate is satisfied. 
\end{proof}

\emph{Step 2:} We now relate this condition with the threshold. First of all, observe that
\begin{equation}
\label{equation:miniclaim}
\frac{1}{C} \| \dd_x \varphi_{-T}|_{E^u} \| \leq |\Lambda_{\mathcal{C}_0}(\varphi_{-T}(x),T)| = \sup_{\xi\in \mathcal{C}_0(x)} \frac{|\xi|}{ |(\dd_{\varphi_{-T}(x)} \varphi_T)^{\top} \xi|} \leq C \| \dd_x \varphi_{-T}|_{E^u} \|.
\end{equation}
Indeed, for $(x,\xi)\in \mathcal{C}_0$, we can decompose $\xi = \xi_s + \xi_{u0}$, with $\xi_s \in E^*_s$, and $\xi_{u0}\in E^\ast_0 \oplus E^\ast_u$. Since $\mathcal{C}_0$ is a small cone around $E^\ast_s$, we have $|\xi_{u0}| < |\xi_s|/2$. Next, we observe that for $T>0$ large enough (but uniformly in $x$), writing:
\[
\dd_{\varphi_{-T}(y)}\varphi_T^{\top}\xi  = \dd_{\varphi_{-T}(y)}\varphi_T^{\top}\xi_s + \dd_{\varphi_{-T}(y)}\varphi_T^{\top}\xi_{u0},
\]
we see that the second term in bounded in norm by:
\[
|\dd_{\varphi_{-T}(x)}\varphi_T^{\top}\xi_{u0}| \leq C |\xi_s| \leq 1/2 \times |\dd_{\varphi_{-T}(x)}\varphi_T^{\top}\xi_s|,
\]
as the last quantity increases exponentially fast. Hence, there exists $C > 0$ such that 
\[
\sup_{\xi\in E_s^*(x)} \frac{|\xi|}{ |(\dd_{\varphi_{-T}(x)} \varphi_T)^{\top} \xi|} \leq \sup_{\xi\in \mathcal{C}_0(x)} \frac{|\xi|}{ |(\dd_{\varphi_{-T}(x)} \varphi_T)^{\top} \xi|}  \leq C \sup_{\xi\in E_s^*(x)} \frac{|\xi|}{ |(\dd_{\varphi_{-T}(x)} \varphi_T)^{\top} \xi|}.
\]
To prove \eqref{equation:miniclaim}, it remains to observe that 
\[
\sup_{\eta \in E_s^*(\varphi_{-T}(x))} \dfrac{|\dd_{\varphi_{-T}(x)} \varphi_T^{-\top} \eta|}{|\eta|} = \sup_{\substack{\eta \in E_s^\ast(\varphi_{-T}(x))\\ |\eta|=1}}  \sup_{\substack{Z\in E^u(x)\\ |Z|=1}} \langle \dd_{\varphi_{-T}(x)} \varphi_T^{-\top}\eta, Z\rangle =  \sup_{Z \in E_u(x)} \dfrac{|\dd_x \varphi_{-T}(Z)|}{|Z|}.
\]

We turn now to the variational interpretations of the quantities above. We introduce:
\[
K' := \sup_{x \in \M} \lim_{T\to + \infty} \frac{1}{T}\log \left(M(T,x) \times \| \dd_x \phi_{-T}|_{E^u} \|^\rho\right) \leq K < 0,
\]
and we want to invert the limit with the sup and show that $K'=K$ actually. Denoting $w(T,x):= \log\left( M(T,x) \| (\dd_x \varphi_{-T})_{|E^u} \|^\rho\right)$, we observe that $w$ is subadditive, in the sense that for all $x \in \M$ and $T_1,T_2 \geq 0$:
\[
w(T_1+T_2, x) \leq  w(T_1,x) + w(T_2, \varphi_{T_1}(x)). 
\]
It therefore satisfies the assumption of Lemma \ref{lemma:sub} and thus $K'=K$. We then set:
\[
\omega(\X) := \inf \left\{\rho > 0 ~|~ \sup_{x \in \M} \lim_{T\to + \infty} \frac{1}{T}\log \left(M(T,x) \times \| \dd_x \phi_{-T}|_{E^u} \|^\rho\right) < 0 \right\},
\] 
so that any $\rho > \omega(\X)$ satisfies $K < 0$ (and thus Lemma \ref{lemma:exponential-decay-estimate} is satisfied). \\

\end{proof}

\subsection{The regularity bootstrap}

We can now complete the proof of Theorem \ref{theorem:source}.

\begin{proof}[Proof of Theorem \ref{theorem:source}]
When $u$ is smooth, Theorem \ref{theorem:source} is a straightforward consequence of Theorem \ref{theorem:sc-source} by taking a fixed value of $h_0$. The only non-trivial statement in Theorem \ref{theorem:source} is therefore the bootstrap statement, asserting that the source estimate still holds for $u$ such that $Au \in C^{\rho_0}_*$ and $B \X u \in C^\rho_*$, with $\omega(\X) < \rho_0 < \rho$.

We let $u_{h} := \Op_{h}(\Psi)u$ where $\Psi \in C^\infty_{\mathrm{comp}}(T^*\M)$ is supported in $\left\{|\xi| \leq 3\right\}$, constant equal to $1$ on $\left\{|\xi| \leq 2\right\}$, and so that $u_{h} \to_{h \to 0} u$ in $C^{\rho_0}_*$ and each $u_{h}$ is smooth. Applying the estimate \eqref{equation:sc-source} of Theorem \ref{theorem:sc-source} with $h_0 \simeq 1$, we have: 
\[
\| A_1 u_{h}\|_{C^\rho_\ast} \leq  C \| B_1 \X u_{h} \|_{C^{\rho}_\ast} + C  \|u_{h}\|_{C^{-N}_*},
\]
where $A_1$ and $B_1$ are $0$-homogeneous and chosen to have smaller support than $A$ and $B$ (and without loss of generality, we assume that $A$ and $B$ are microlocally equal to $\mathbbm{1}$ on the wavefront set of $A_1$ and $B_1$). We write
\[
\X u_{h} = \X \Op_{h}(\Psi) u = \Op_{h}(\Psi) \X u + \underset{\in  \Psi_{h}^0}{\underbrace{ [\X, \Op_{h}(\Psi)]}} u. 
\]

First of all, we claim that under the assumption that $B \X u \in C^\rho_*$, we have 
\[
\|B_1 \Op_{h}(\Psi) \X u\|_{C^\rho_*} \leq C \| B \X u\|_{C^\rho_*} + C \| u \|_{C^{-N}},
\]
for some $C > 0$ independent of $h$. Indeed, this follows from:
\[
B_1 \Op_{h}(\Psi) \X u = B_1 \Op_{h}(\Psi) B \X u +  B_1 \Op_{h}(\Psi) (\mathbbm{1}-B) \X u,
\]
It suffices to observe that $B_1$ and $\mathbbm{1}-B$ have disjoint microsupport, and use the uniform boundedness on $C^\rho_\ast $ of $B_1 \Op_h(\Psi)$ as $h\to 0$ (by Theorem \ref{theorem:cv}).

Now, we observe that
\[
R_{h}:=B_1 [\X, \Op_{h}(\Psi)] \in \Psi^0_{h}(\M)
\]
is $h$-semi-classical, with $\WF_{h}(R_{h}) \subset\{1/2 < |\xi| < 1\}$. According to Lemma \ref{lemma:bound-composition}, $\WF_h(R_h)$ is contained in $\WF(B_1)$, which is itself contained in a small conical neighbourhood of $E^\ast_s$. We can thus apply Theorem \ref{theorem:sc-source} with $h_0 = h$, and $A_h = R_h$. We obtain the existence of $\widetilde{B}_{h}$ (which can be chosen with semiclassical microsupport contained inside the classical wavefront set of $B$) such that:
\[
\|R_{h}u \|_{C^{\rho_0}_\ast} \leq C \| \widetilde{B}_{h} \X u \|_{C^{\rho_0}_\ast} + C h^N \|u\|_{C^{-N}_*},
\] 
As $R_{h}$ is compactly microsupported, we can use the second case of Lemma \ref{lemma:comparaison-rho-rho'} to compare $C^\rho_*$ and $C^{\rho_0}_*$-norms. We deduce that (the integer $N$ might be different from one line to another as some $h^\rho$ might be absorbed in it):
\[
\begin{split}
\|R_{h} u\|_{C^\rho_*} & \lesssim h^{\rho_0-\rho} \|R_{h}u \|_{C^{\rho_0}_*} + h^N\|u\|_{C^{-N}} \lesssim h^{\rho_0-\rho} \| \widetilde{B}_{h} \X u \|_{C^{\rho_0}_\ast} + h^N \|u\|_{C^{-N}_*}.
\end{split}
\]
Now, we observe that applying now the first case of Lemma \ref{lemma:comparaison-rho-rho'} (since $\widetilde{B}_h$ is not microsupported at $0$), we have:
\[
\|\widetilde{B}_{h} \X u \|_{C^{\rho_0}_\ast} \lesssim h^{\rho-\rho_0} \|\widetilde{B}_{h} \X u \|_{C^{\rho}_\ast} + h^N \|u\|_{C^{-N}_*},
\]
and thus: $\|R_{h} u\|_{C^\rho_*} \lesssim \|B \X u \|_{C^\rho_*} + \|u\|_{C^{-N}*}$. Combining all the estimates, we find that there exists a $h$-independent constant $C > 0$ such that:
\begin{equation}
\label{equation:a1}
\|A_1 u_{h} \|_{C^{\rho}_\ast} \leq C \left( \|B  \X u\|_{C^\rho_*} +  \|u\|_{C^{-N}_*}\right).
\end{equation}
Observe that $A_1 u_h \rightarrow A_1 u$ in the sense of distributions. Since $C^\rho_\ast$ satisfies the Fatou property \cite[Proposition 2, page 15]{Runst-Sickel-96}, it implies that $A_1 u \in C^\rho_\ast$ and
\[
\|A_1 u\|_{C^{\rho}_\ast} \leq C \left( \|B  \X u\|_{C^\rho_*} +  \|u\|_{C^{-N}_*}\right),
\]
for some possibly different constant $C > 0$. Finally, in order to obtain an estimate on $Au$ instead of $A_1 u$ (where $A$ has slightly larger wavefront set than $A_1$), it suffices to use standard propagation of singularities.

\end{proof}

\subsection{Regularity of solutions to cohomological equations}
\label{sec:proof-main}

Here we prove the main Theorem \ref{theorem:regularity}.

\begin{proof}[Proof of Theorem \ref{theorem:regularity}]
From the elliptic considerations already mentionned in the introduction, we know that $\WF(u) \subset E^\ast_u \oplus E^\ast_s$. Next, since $\rho > \omega_+(\X)$, we can use Theorem \ref{theorem:source} to deduce that for any $s>0$, and a pseudo-differential operator of order $0$ microsupported near $E^\ast_s$, $Au$ is in $C^s_\ast$. In particular, $A u\in C^\infty$. Choosing $A$ to be elliptic near $E^\ast_s$, we deduce that $E^\ast_s \cap \WF(u) = \emptyset$. The same argument applied with $E^\ast_u$ in reversed time (which is allowed since $\rho > \omega_-(\X)$) implies that $E^\ast_u \cap \WF(u) = \emptyset$. 

Now, we assume that there exists $(x,\xi)\in \WF(u)$. Using Proposition \ref{prop:usual-propagation-Cs}, we deduce that $\{ \Phi_t(x,\xi) \ |\ t\in \R\}\subset \WF(u)$. However, in $E^\ast_u \oplus E^\ast_s$, the dynamics of $\Phi_t$ is quite simple to describe: if $\xi\neq 0$, either $\Phi_t(x,\xi) \to E^\ast_u \cap \partial T^\ast \M$ as $t\to + \infty$, or $\Phi_t(x,\xi) \to E^\ast_s\cap \partial T^\ast \M$ as $t\to -\infty$. Since the wavefront set is closed and does not intersect $E^\ast_s \cup E^\ast_u$, we deduce that there can be no such point $(x,\xi)$. Since $\WF(u) = \emptyset$, $u$ is smooth.
\end{proof}

\section{Regularity in hyperbolic dynamics}

\label{section:cohomological}

\subsection{Livsic theory}

We denote by $\mc{G}$ the set of periodic orbits of the flow. The \emph{X-ray transform} operator $I$ appears in several geometric and dynamical problems: it consists in integrating functions along periodic orbits:
\begin{equation}
\label{equation:xray}
I : C^\alpha(\M) \rightarrow \ell^\infty(\mc{G}), ~~~~ \mc{G} \ni \gamma \mapsto If(\gamma) := \dfrac{1}{\ell(\gamma)} \int_0^{\ell(\gamma)} f(\varphi_t(x)) \dd t,
\end{equation}
where $x \in \gamma$ is arbitrary. The kernel of the X-ray transform on $C^\alpha(\M)$ was first characterized by Liv\v{s}ic \cite{Livsic-72}:

\begin{theoremexp}[\cite{Livsic-72}]
Let $X$ be a transitive Anosov vector field. Let $f \in C^\alpha(\M)$ be a function such that $If = 0$. Then, there exists $u \in C^\alpha(\M)$ such that $f = Xu$. Moreover, $u$ is unique modulo an additive constant.
\end{theoremexp}

We call \emph{coboundaries} the functions of the form $Xu$, and we call \emph{cohomological equation} an equality of the form $f =Xu$. It is also natural to deal with other regularities, namely if $f \in C^\infty(\M)$, then one expects that $u \in C^\infty(\M)$. This was proved by de la Llave-Marco-Moriy\'on \cite{DeLaLlave-Marco-Moryon-86} but the proof uses more sophisticated tools than \cite{Livsic-72}; it relies on the Journé Lemma \cite{Journe-86}.

\begin{theoremexp}[\cite{DeLaLlave-Marco-Moryon-86}]
\label{theorem:livsic-smooth}
Assume $X$ is transitive. For $k \in \left\{1,2,...,+\infty\right\}$, let $f \in C^k(\M)$ be a function such that $If=0$. Then, there exists $u \in C^k(\M)$ such that $f=Xu$.
\end{theoremexp}

We are studying flows here, but it may worthwhile to consider that similar problems and results can be formulated in the case of Anosov \emph{diffeomorphisms} $F : \M \rightarrow \M$ (see \cite{Hasselblatt-Katok-95}). 

The Liv\v{s}ic's theorem can be extended in several directions. Notably:
\begin{itemize}
	\item If $If \geq 0$ and $f \in C^\alpha(\M)$, one can decompose $f$ as $f = Xu + h$ for some $h \geq 0$. This result of \cite{Lopes-Thieullen-05} comes with a control $\|h\|_{C^{\alpha}} \leq C \|f\|_{C^\alpha}$ with $0<\beta < \alpha$. 
	\item if $If = \mc{O}(\eps)$ and $f \in C^\alpha(\M)$, one can decompose $f$ as $f = Xu+h$, with again a control of the form $\|h\|_{C^\beta} \leq C \eps^\gamma$, for some $\gamma > 0$ and $0<\beta < \alpha$. This is a recent result of \cite{Gouezel-Lefeuvre-19}.
	\end{itemize}
In both cases, the statement is only available with values of $\beta$ strictly smaller than $1$. The question of obtaining the control with $\beta = \alpha$, even for small $\alpha$, is open. 

One can also consider a more general setting involving (trivial) principal bundles. Let $G$ be a Lie group and let $C$ be a smooth cocycle, namely a map $C : \mc{M} \times \R \rightarrow G$ such that
\[
C(x,t+s) = C(\varphi_t (x), s)C(x,t)
\]
We say that $C$ satisfies the \emph{periodic orbit obstruction} if for every periodic point $x \in \M$ (of period $T$), one has $C(x,T) = e_G$. Liv\v{s}ic \cite{Livsic-72} characterized such cocycles (when the bundle is trivial) in Hölder regularity:

\begin{theoremexp}[\cite{Livsic-72}]
Assume $X$ is transitive, $G$ is compact, $C : \mc{M} \times \R \rightarrow G$ is a $\alpha$-Hölder continuous cocycle satisfying the periodic orbit obstruction. Then, there exists a $\alpha$-Hölder continuous map $u : \M \rightarrow G$ such that for all $x \in \M, t \in \R$: $C(x,t) = u(\varphi_t(x))u(x)^{-1}$. Moreover, $u$ is unique modulo multiplication by a constant element $g_0 \in G$.
\end{theoremexp}

Once again, proving the smooth regularity of $u$ when $C$ is smooth is a more difficult problem. This was solved by \cite[Theorem 2.4]{Nitica-Torok-98}:

\begin{theoremexp}[\cite{Nitica-Torok-98}]
\label{theorem:livsic-smooth-cocycle}
Assume that $G$ is a closed subgroup in a finite-dimensional Lie group, $C : \mc{M} \times \R \rightarrow G$ is a smooth cocycle such that there exists a $\alpha$-Hölder (for some $\alpha > 0$) $u : \M \rightarrow G$ such that for all $x \in \M, t \in \R$: $C(x,t) = u(\varphi_t(x))u(x)^{-1}$. Then $u$ is smooth.
\end{theoremexp}

The transitivity assumption is not needed and $G$ need not be compact here. Note that the trivial vector bundle $\M \times \C$ can be interpreted in the framework of the previous Liv\v{s}ic cocycle theorem by introducing the cocycle
\[
C(x,t) := \exp\left(\int_0^t f(\varphi_s(x)) \dd s \right).
\]
Then, $If = 0$ if and only if $C$ satisfies the periodic orbit obstruction in the Lie Group $(\R^+_*,\times)$. One can extend the definition of a cocycle to include the case of the parallel transport map induced by a (complex) vector bundle (of rank $r$) endowed with a connection $(\mc{E},\nabla)$ over $\M$.

\subsection{Regularity of cohomological equations via source estimates}

%
%
%

We now explain how the previous regularity results can be retrieved from Theorem \ref{theorem:source}. We start with a statement we did not find in the literature:

\begin{theorem}
\label{theorem:smoothness}
Let $\nabla^{\E}$ be a unitary connection on a vector bundle $\E \rightarrow \M$. Then $\omega(\nabla^{\E}_X)=0$. In particular, if $f \in C^\infty(\M)$, $u \in C^\alpha(\M)$ for some $\alpha > 0$ and $\nabla^{\E}_X u = f$, then $u \in C^\infty(\M)$.
\end{theorem}

Theorem \ref{theorem:smoothness} has important consequences as we shall see. It is also reinvested in subsequent papers, see \cite{Cekic-Lefeuvre-21-1, Cekic-Lefeuvre-21-2}.

\begin{remark}
Note that the fact that $\nabla^{\E}_X$ is unitary is actually not necessary, but in the non-unitary case the exponent $\omega(\nabla^{\E}_X)$ might be strictly positive and then $\alpha$ has to be taken large enough in order to enforce the threshold condition \eqref{equation:threshold}.
\end{remark}

\begin{proof}
The fact that $\omega(\nabla^{\E}_X) = 0$ is a straightforward consequence of the definition of the threshold \eqref{equation:threshold} and the fact that the propagator $\|e^{t\nabla^{\E}_X}\|_{\mc{L}(L^\infty,L^\infty)} = 1$ is bounded, independently of $t \in \R$ since the connection is unitary. Now it suffices to apply Theorem \ref{theorem:regularity}. 
\end{proof}

We conclude with the proof of the general smooth Liv\v{s}ic cocycle theorem for transparent cocycles with values in a finite-dimensional Lie group $G$ ($G$ needs not be compact). We retrieve \cite[Theorem 2.4]{Nitica-Torok-98}:

\begin{theorem}
\label{theorem:nt}
Let $G$ be a finite-dimensional Lie group. Let $C : \M \times \R \rightarrow G$ be a smooth cocycle such that there exists $u \in C^\alpha(\M,G)$ (for some $\alpha > 0$) such that $C(x,t)=u(\varphi_t(x))u(x)^{-1}$, for all $x \in \M, t \in \R$. Then $u$ is smooth.
\end{theorem}

\begin{remark}
It is very likely that the infinite-dimensional case (e.g. cocycles with values in the group of diffeomorphisms) could also be treated with our methods by deriving source estimates with values in vector bundles whose fibers are Banach or Hilbert spaces. As a starting point, this would require to study pseudodifferential operators acting on infinite-dimensional vector spaces. This is left to future investigation.
\end{remark}

\begin{proof}
First of all, let us assume that $G$ is a \emph{linear Lie group}, i.e. it embeds into a $\mathrm{GL}_r(\R)$ for some $r \geq 0$. Writing
\[
\left. \dfrac{d}{dt} C(x,t)\right|_{t=0} =: A(x) \in C^\infty(\M,\mathfrak{g}),
\]
and using the fact that the group is linear, we obtain that $u$ satisfies the equation
\begin{equation}
\label{equation:simplification0}
(-X+A)u = 0.
\end{equation}
The operator $-X+A$ acts on $C^\infty(\M,\C^r) \rightarrow C^\infty(\M,\C^r)$ (where $r$ is such that $G \hookrightarrow \mathrm{GL}_r(\C)$). If $(\mathbf{e}_1, ..., \mathbf{e}_r)$ denotes a basis of $\C^r$, it is sufficient to show that $u \cdot \mathbf{e}_i$ is smooth (for any $i=1,...,r$). But $u \cdot  \mathbf{e}_i$ is Hölder-continuous and satisfies $(-X+A)(u \cdot  \mathbf{e}_i) = 0$. We are thus in the setting of Theorem \ref{theorem:source}. We claim that $\omega(-X+A)=0$. This follows from the observation that the propagator $U(t)$ of $-X+A$ acting on $L^\infty(\M,\C^r) \rightarrow L^\infty(\M,\C^r)$ is bounded by a constant independent of $t \in \R$. Now, a direct computation shows that $U(t) = e^{-tX} C(\cdot,t)$ and since $C(x,t) = u(\varphi_t x)u(x)^{-1}$, the bound on $L^\infty(\M,\C^r) \rightarrow L^\infty(\M,\C^r)$ is immediate.

In the general case, where $G$ might not be a linear Lie group, we use Ado's Theorem (see \cite[Conclusion 5.26]{Hall-15}): any Lie group covers a linear group. In other words, there exists a projection $\pi : G \rightarrow G_0$ (which is a local diffeomorphism) such that $G_0$ is linear. Let $u_0 := \pi u$ be the projection of the cocycle (which is also Hölder-continuous). As it $u_0$ satisfies the equation \eqref{equation:simplification0}, it is smooth. Hence $u$ is also smooth since $\pi$ is locally a diffeomorphism.
\end{proof}

\subsection{Rigidity of the foliation}

In this paragraph, we discuss the rigidity of the foliation of Anosov flows and show that they fit into the framework of radial estimates. In \cite{Hasselblatt-92}, Hasselblatt showed that if the Anosov splitting is smoother than some universal function of the constants entering in the definition \eqref{equation:anosov} of the Anosov flow, then the splitting is actually smooth. We will show below that this can be obtained as a straightforward consequence of radial source estimates. In the particular case of contact Anosov flows, sharper results have been obtained and we refer to \cite{Hurder-Katok-90, Benoist-Foulon-Labourie-90, Benoist-Foulon-Labourie-92}. As far as this paragraph is concerned, we will not assume any particular extra feature for the flow.

We denote by $\alpha$ the exponent of Hölder regularity of the foliation by strong stable/unstable leaves. Let $\pi_{E_u}, \pi_{E_s} \in C^\alpha(\M, T\M \otimes T^*\M)$ be the projection onto $E_u$ (resp. $E_s$), parallel to $E_s \oplus \R X$ (resp. $E_u \oplus \R X$) whose regularity is given by that of the foliation.

\begin{lemma}
\label{lemma:projecteurs}
$\mc{L}_X \pi_{E_u} =  \mc{L}_X \pi_{E_s} = 0.$
\end{lemma}

\begin{proof}
We denote by $d_s$ and $d_u$ the respective dimensions of $E_s$ and $E_u$. We fix a point $x_0$ and consider a local basis $U_1,...,U_{d_u}$ of $E_u$ and a local basis $S_1,...,S_{d_s}$ of $E_s$. We define the dual covectors by $U_i^*(U_i) = 1$ and for $j \neq i, k = 1,...,d_s$, $U_i^*(U_j)= U_i^*(X) = U_i^*(S_k) = 0$ (and $S_i^*$ is defined similarly). Then, locally around $x_0$, we have:
\[
 \pi_{E_u} = \sum_{i=1}^{d_u} U_i \otimes U_i^*, \hspace{2cm} \pi_{E_s} = \sum_{i=1}^{d_s} S_i \otimes S_i^*.
\]
Moreover, as $E_u$ and $E_s$ are invariant by the flow, we can write
\[
\mc{L}_X U_i(x) = U(x) U_i(x), \hspace{2cm} \mc{L}_X S_i(x) = S(x) S_i(x),
\]
for some $\alpha$-Hölder continuous matrices $U \in C^\alpha(\M, E_u), S \in C^\alpha(\M,E_s)$. Using the compatibility of the Lie derivative with the contractions, we obtain from $\mc{L}_X (U_i^*(U_i)) = \mc{L}_X 1=0$ (and the other relations) that:
\[
\mc{L}_X U_i^*(x) = -U(x)^\top U_i^*(x), \hspace{2cm} \mc{L}_X S^*_i(x) = -S^\top(x) S_i^*(x),
\]
where ${}^{\top}$ denotes the transpose. Hence:
\[
\mc{L}_X  \pi_{E_u} = \sum_{i=1}^{d_u} (\mc{L}_XU_i) \otimes U_i^* + U_i \otimes \mc{L}_X U_i^* = \sum_{i=1}^{d_u} (U \cdot U_i) \otimes  U_i^* - U_i \otimes (U^\top \cdot U_i^*) = 0.
\]
\end{proof}

In the following, $\X := \mc{L}_X$ is the Lie derivative acting on $T\M \otimes T^*\M$. Let us make the threshold $\omega(\X)$ more explicit. First of all, observe that
\[
e^{-t \X} : T_{\varphi_{-t}(x)}\M \otimes T^*_{\varphi_{-t}(x)}\M \rightarrow  T_{x}\M \otimes T^*_{x}\M
\]
is equal to the tensor map $\dd_{\varphi_{-t}(x)} \varphi_t \otimes \dd_{\varphi_{-t}(x)} \varphi_t^{-\top}$, where ${}^{-\top}$ denotes the inverse transpose. Hence:
\[
M(t,x) \leq \|\dd_{\varphi_{-t}(x)} \varphi_t\|\cdot\|\dd_{\varphi_{-t}(x)} \varphi_t^{-\top}\| \leq C  \|\dd_{\varphi_{-t}(x)} \varphi_t|_{E_u}\| \|\dd_{\varphi_{-t}(x)} \varphi_t^{-\top}|_{E_u^*}\|.
\]
Moreover, a simple calculation shows that
\[
\|\dd_{\varphi_{-t}(x)} \varphi_t^{-\top}|_{E_u^*}\| = \|\dd_x \varphi_{-t}|_{E_s}\|,
\]
(see also the end of the proof of Theorem \ref{theorem:source} where such an equality is derived). Hence:
\begin{equation}
\label{equation:cal}
M(t,x) \|\dd_x \varphi_{-t}|_{E^u}\|^\rho \leq C  \|\dd_{\varphi_{-t}(x)} \varphi_t|_{E_u}\| \|\dd_x \varphi_{-t}|_{E_s}\| \|\dd_x \varphi_{-t}|_{E^u}\|^\rho.
\end{equation}
This enables us to recover the classical result of \cite{Hasselblatt-92}:

\begin{theorem}
Assume $\M$ is $3$-dimensional and $X$ is volume preserving. Then $\omega(\X)\leq 2$. In particular, if the foliation is $C^{2+\delta}_*$-smooth, for some $\delta > 0$, then it is smooth.
\end{theorem}

In the contact three-dimensional case, a sharper result was obtained by \cite{Hurder-Katok-90}.

\begin{proof}
Let $\mu$ be the smooth measure on $\M$ preserved by $X$. We choose a Riemannian metric $g$ on $\M$ so that the measure induced by $g$ is equal to $\mu$. Fix $x \in \M$ and let $v_{s,u} \in E_{s,u}(x)$ of norm $1$ (with respect to $g$). There exist functions $r_{s,u} : \M \rightarrow \R$ such that:
\[
|\dd_x \varphi_T(v_{s,u})| = \exp \left( \int_0^T r_{s,u}(\varphi_t(x))  \dd t\right),
\]
and these converge respectively to $0$ for $s$ and $+\infty$ for $u$ as $T \rightarrow \infty$ by the Anosov property \eqref{equation:anosov}. Also note that:
\[
|\det \dd_x \varphi_T| = 1 = \exp \left( \int_0^T r_{s}(\varphi_t(x)) + r_{u}(\varphi_t(x))  \dd t\right),
\]
that is
\begin{equation}
\label{equation:det}
\int_0^T r_{s}(\varphi_t(x)) + r_{u}(\varphi_t(x))  \dd t = 0.
\end{equation}
This gives using \eqref{equation:cal}
\[
\begin{split}
|M(T,x)| \|\dd_x \varphi_{-T}|_{E^u}\|^\rho  \leq C \exp\left( \int_{-T}^0 \left[(1-\rho) r_u(\varphi_t x) - r_s(\varphi_t(x))\right] \dd t\right).
\end{split}
\]
Taking $\rho = 2$, we see that this quantity is bounded using \eqref{equation:det} and it converges to $0$ as $T \rightarrow \infty$ for any $\rho > 2$, that is $\omega(\X) \leq 2$. By Lemma \ref{lemma:projecteurs}, the conclusion of the Theorem is immediate.
\end{proof}

Let us consider now the general case when $X$ does not preserve a volume. We can define global Lyapunov exponents $\lambda^{\max,\min}_{u,s}$ so that the following inequalities hold for all $t \geq 0$:
\[
\frac{1}{C} e^{-\lambda^{\max}_u t} \leq \| \dd\varphi_{-t}|_{E_u}\| \leq Ce^{-\lambda^{\min}_u t}, ~~  \frac{1}{C} e^{-\lambda^{\max}_s t} \leq \| \dd\varphi_{t}|_{E_s}\| \leq Ce^{-\lambda^{\min}_s t},
\]
where $C > 0$ is uniform. We then obtain using \eqref{equation:cal}:

\begin{theorem}
\label{theorem:rigidity-foliation}
Let $\X := \mc{L}_X$ be the Lie derivative acting on $T\M \otimes T^*\M$. Then:
\[
\omega(\X) \leq \dfrac{\lambda^{\max}_u + \lambda^{\max}_s}{\lambda^{\min}_u}.
\]
In particular, if the stable and unstable foliation is $C^\rho$-regular, where
\[
\rho > \max\left(\dfrac{\lambda^{\max}_u + \lambda^{\max}_s}{\lambda^{\min}_u},  \dfrac{\lambda^{\max}_u + \lambda^{\max}_s}{\lambda^{\min}_s} \right),
\]
then it is smooth.
\end{theorem}

We could also obtain a sharper statement as in \cite{Hasselblatt-92} by letting the exponents depend on the point $x$. The threshold would then involve a supremum over all $x \in \M$. We also notice that similar results have just been obtained by \cite{Guillarmou-Poyferre-21} using paradifferential calculus.

\section{Stability estimates for the marked length spectrum}
\label{section:mls}

\subsection{Geodesic stretch, main result}

From now on, $SM := SM_{g_0}$ and the metric $g_0$ is fixed. We denote by $\otimes^2_S T^*M \rightarrow M$ the vector bundle of symmetric $2$-tensors on $M$, and we let 
\[
\pi_2^* : C^\infty(M,\otimes^2_S T^*M) \rightarrow C^\infty(SM), ~~~ \pi_2^* f(x,v) := f_x(v,v). 
\]
We denote by ${\pi_2}_*$ the formal adjoint of $\pi_2^\ast$, and by $\nabla$ the Levi-Civita connection of $g_0$. Then we define
\[
D := \mathcal{S} \nabla : C^\infty(M,T^*M) \rightarrow C^\infty(M,\otimes^2_S T^*M),
\]
where $\mc{S} : \otimes^2 T^*M \rightarrow \otimes^2_S T^*M$ denotes the symmetrization operator. Let $D^* = -\Tr(\nabla \cdot)$ be the formal adjoint of $D$. Any symmetric $2$-tensor $f \in C^\infty(M,\otimes^2_S T^*M)$ can be uniquely decomposed as
\[
f = Dp + h,
\]
where $D^*h = 0$ and $p \in C^\infty(M,T^*M)$. We say that $h$ is the \emph{solenoidal part} of $f$ and $Dp$ is the \emph{potential part}.

It is a classical result that for $c \in \mc{C}$, a free homotopy class and $g_0, g\in \mathrm{Met}_{\mathrm{An}}$,
\begin{equation}\label{eq:DL-MLS-ordre1}
\frac{L_g(c)}{L_{g_0}(c)} - 1 = \frac{1}{L_{g_0}(c)} \int_{\gamma_{g_0}(c)} \pi_2^\ast(g-g_0) ~\dd \gamma_{g_0}(c) + o( \| g - g_0 \|_{C^2} ).  
\end{equation}
The X-ray transform $I_2^{g_0}$ associated with $g_0$ is exactly the map which takes a symmetric two-tensor $f$ to 
\[
I_2^{g_0}f := I \circ \pi_2^* f
\]
defined on the free homotopy classes, where $I$ is the X-ray transform on $\M := SM$ defined in \eqref{equation:xray}. The relation \eqref{eq:DL-MLS-ordre1} proves that $g-g_0$ controls $L_g- L_{g_0}$. However for the purpose of studying the Burns-Katok conjecture, it is desirable to control $g-g_0$ by $L_g-L_{g_0}$. To find such a control, the first question that arises is whether $I^{g_0}_2 (g-g_0)$ controls $g-g_0$. It was the gist of \cite{Guillarmou-Lefeuvre-18} that this is essentially the \emph{only} obstruction. However the estimates in \cite{Guillarmou-Lefeuvre-18} required a very large number of derivatives, and our purpose here is to obtain essentially the same result but with less regularity. 

Potential tensors are always in the kernel of the X-ray transform due to the relation $\pi_2^*D = X \pi_1^*$ (where $\pi_1^* p(x,v) := p_x(v)$). It is customary to say that $I^{g_0}_2$ is \emph{injective} (or \emph{s-injective}) if for every $f\in C^\infty(M, \otimes^2_S T^\ast M) \cap \ker D^*$, if $I^{g_0}_2 f = 0$, then $f=0$. It turns out that for $g_0\in \mathrm{Met}_{\mathrm{An}}$, $I^{g_0}_2$ is injective provided:
\begin{itemize}
	\item $M$ is a surface \cite{Paternain-Salo-Uhlmann-14-1,Guillarmou-17-1},
	\item $\dim(M)\geq 3$ and $g_0$ has non-positive curvature \cite{Croke-Sharafutdinov-98},
	\item $\dim(M)\geq 3$ and $g_0$ is $C^k$-generic for some $k\gg1$ \cite{Cekic-Lefeuvre-21}.
\end{itemize}
It is conjectured that $I^{g_0}_2$ is actually injective for every $g_0 \in \mathrm{Met}_{\mathrm{An}}$. 

Instead of working directly with $L_g/L_{g_0}-1$ as in \cite{Guillarmou-Lefeuvre-18}, we will follow the tactic initiated in \cite{Guillarmou-Knieper-Lefeuvre-19}. Let $g_0, g\in \mathrm{Met}_{\mathrm{An}}$. It is known (see \cite[Appendix B]{Guillarmou-Knieper-Lefeuvre-19} for instance) that there exists an orbit-conjugacy of the geodesic flows i.e. a map
\[
\psi_g : SM=SM_{g_0} \rightarrow SM_g
\]
such that
\[
\dd \psi_{g}(X_{g_0}(z)) = a_g(z) X_g(\psi_g(z)), ~~~  \forall z \in SM,
\]
where $a_g$ is a function on $SM$ called the \emph{infinitesimal stretch}. The maps $\psi_g$ and $a_g$ are just $C^\nu$ for some $\nu > 0$. The map $\psi_g$ is not unique and $a_g$ is only defined up to a coboundary, namely a term of the form $X_{g_0} u$. The infinitesimal stretch is linked to the marked length spectrum by the following equality: for all $c \in \mc{C}$,
\[
L_g(c) = \int_{\gamma_{g_0}(c)} a_g(\varphi_t(z)) \dd t,
\]
where $z$ is an arbitrary point on $\gamma_{g_0}(c)$. The following lemma is well-known (see the discussion in \cite[Section 2.5]{Guillarmou-Knieper-Lefeuvre-19} for instance):

\begin{lemma}
The following statements are equivalent:
\begin{enumerate}
\item $L_g = L_{g_0}$,
\item The geodesic flows are conjugate i.e there exists $\psi_g$ a H\"older homeomorphism with $\psi_g \circ \varphi_t^{g_0} = \varphi_t^g \circ \psi_g$, for all $t \in \R$,
\item $a_g$ is cohomologous to the constant function $\mathbf{1}$.
\end{enumerate}
\end{lemma}
Since $a_g$ is only defined up to coboundaries, is it more sensible to measure its size modding out those coboundaries. More precisely, given $\alpha \in \R_+ \setminus \N$, we introduce the space of coboundaries $D^\alpha$ of regularity $\alpha$, namely:
\[
D^\alpha(SM) := \Big\{ X_{g_0} u ~\Big|~ u \in C^\alpha(SM_{g_0}), X_{g_0} u \in C^\alpha(SM) \Big\}.
\]
This is a closed subspace of $C^\alpha(SM)$ and we can therefore consider the quotient space $C^\alpha(SM)/D^\alpha(SM)$ endowed with the natural norm 
\[
\|[f]\|_{C^\alpha/D^\alpha} := \inf\Big\{ \|f + X_{g_0}u\|_{C^\alpha}\ \Big|  X_{g_0} u \in D^\alpha \Big\},
\]
where $[f]$ denotes an element in $C^\alpha/D^\alpha$. We now restate our main result:

\begin{theorem}
\label{theorem:mlsx}
Let $(M,g_0)$ be a smooth Anosov manifold and further assume $I^{g_0}_2$ is injective. For any $\eps>0$, there exists $\nu, C > 0$ such that the following holds. For any metric $g$ such that $\|g-g_0\|_{C^{3+\eps}} < 1/C$, there exists a $C^{4+\eps}$-diffeomorphism $\phi$, isotopic to the identity, such that
\[
\| \phi^* g - g_0 \|_{C^{\nu-1}} \leq C \inf_{\substack{u \in C^\nu(SM), \\ Xu \in C^\nu(SM)}} \|a_g-\mathbf{1} + Xu\|_{C^{\nu}}.
\]
\end{theorem}

The exponent $\nu > 0$ has to be chosen small enough (proportional to $\eps$) as we shall see in the proof. We believe that this bound could help proving a local rigidity statement for the \emph{unmarked length spectrum} on surfaces. This is left to future investigation. The main novelty here is that this bound is independent of the dimension i.e. the metrics need only to be $C^{3+\eps}$-close whereas in previous works \cite{Guillarmou-Lefeuvre-18,Guillarmou-Knieper-Lefeuvre-19}, this regularity was increasing (linearly) with the dimension\footnote{namely, more than $3\dim(M)/2 + 9$ derivatives, which is worse than $3+\eps$ even for surfaces}. Moreover, the bound in Theorem \ref{theorem:mls} is linear whereas the bounds in \cite{Guillarmou-Lefeuvre-18, Guillarmou-Knieper-Lefeuvre-19} were non linear (but they involved the marked length spectrum directly, however). In particular, one can retrieve the bounds of \cite{Guillarmou-Lefeuvre-18, Guillarmou-Knieper-Lefeuvre-19} by simply applying the approximate Liv\v{s}ic theorem of \cite{Gouezel-Lefeuvre-19}: there exists a $\beta,\tau > 0$ such that $a_g-\mathbf{1} = Xu + h$, where $\|h\|_{C^\beta} \leq C \|L_g/L_{g_0}-\mathbf{1}\|_{\ell^\infty(\mc{C})}^\tau$ for some $\tau>0$. This precisely gives $\|[a_g-\mathbf{1}]\|_{C^\beta/D^\beta} \leq C \|L_g/L_{g_0}-\mathbf{1}\|_{\ell^\infty(\mc{C})}^\tau$. Using Theorem \ref{theorem:mls}, we retrieve the bound:
\[
\| \phi^* g - g_0 \|_{C^{\nu-1}} \leq C \|L_g/L_{g_0}-\mathbf{1}\|_{\ell^\infty(\mc{C})}^\tau.
\]
However, note that the exponent $\beta > 0$ is not well controlled (and might be $\ll 1$ actually), nor is $\tau$.

\subsection{Expansion of the stretch}

We will now study the dependence of the stretch on the metric $g$ close to $g_0$. According to \cite[Proposition 2.2]{Katok-Knieper-Pollicott-Weiss-89}, there exist $\nu_0\in (0,1)$ so that for $k\geq 1$, the map
\[
C^{2+k}(M,\otimes^2_S T^*M) \ni g \mapsto a_g \in C^{\nu_0}(SM),
\]
is $C^k$. One can give a lower bound on $\nu_0$ in terms of Lyapunov exponents of $g_0$. In particular, for $k=3$, it admits a Taylor expansion:
\begin{equation}
\label{equation:taylor}
a_g - \mathbf{1} = 0 + \mathbf{D}_{g}a_g|_{g=g_0}(g-g_0) + \mathbf{D}^2_{g}a_g|_{g=g_0} (g-g_0)^{\otimes 2} + \mc{O}_{C^{\nu_0}}(\|g-g_0\|_{C^5}^3). 
\end{equation}
(Here, distinct from the notation $\dd f$ which is the differential of a function $f$ on $SM$, the notation $\mathbf{D}_g a_{g} |_{g=g_0}(h)$ denotes the differential of the map $g \mapsto a_g$ at the point $g=g_0$, applied to the two-tensor $h$. By $\mathbf{D}_g a_{g} (h)$ we will denote the value of the differential at the point $g$.)

If we plugged this estimate in our machinery, we would obtain a result using $8-\eps$ derivatives on the metric, for some $\eps>0$ (and $9-\eps$ derivatives using the Taylor expansion to order $1$ instead). However, it turns out that modding out coboundaries, we can drastically improve the regularity of the stretch map:

\begin{proposition}\label{prop:sharp-estimate-stretch}
For $\eps>0$, there exists $C,\nu>0$ small enough such that:
\begin{equation}
\|a_g - \mathbf{1} -1/2 \times \pi_2^\ast(g-g_0)\|_{C^\nu/D^\nu} \leq C \|g-g_0\|_{C^{\nu-1}} \|g-g_0\|_{C^{3+\eps}}.
\end{equation}
\end{proposition}

In order to prove Proposition \ref{prop:sharp-estimate-stretch}, we need to introduce some notations and recall some elements of the proof of \cite[Proposition 2.2]{Katok-Knieper-Pollicott-Weiss-89}. The maps $\psi_g$ introduced before take values in $SM_g$. It will be convenient to pullback everything to the same unit tangent bundle $SM = SM_{g_0}$. We thus introduce:
\[
\Phi_g : SM \to SM_g, ~~~~\Phi_{g}(x,v) = (x,v/|v|_g),
\]
and define $Y_g := \Phi_g^* X_g$, where $X_g$ is the geodesic vector field on $SM_{g}$. We also write $\beta_g := \Phi_g^* \alpha_g$, where $\alpha_g$ denotes the Liouville $1$-form on $SM_g$, that is for $(x,v) \in SM_g$ and $\xi \in T_{(x,v)}(SM_g)$, we have:
\begin{equation}
\label{equation:bam}
(\alpha_g)_{(x,v)}(\xi) = g_x(v, \dd\pi_{(x,v)}(\xi)),
\end{equation}
where $\pi : TM \to M$ denotes the projection. Each $1$-form $\beta_g$ comes with a contact distribution $\ker \beta_g = E^u(g) \oplus E^s(g)$ on $SM$. In the computations the following map will appear several times:
\[
\pi_{2,g}^\ast h : SM_{g_0} \owns (x,v) \to \frac{h_x(v,v)}{|v|_g^2}.
\]

Let us now recall the gist of the proof of structural stability. To perturb the vector field $Y_g$, the idea is to consider the map
\[
\Xi: ( Y_{g'}, \Psi , a) \mapsto \dd \Psi(Y_g) - a \times ( Y_{g'}\circ \Psi) \in C^{\nu_g}, 
\]
defined on 
\begin{itemize}
	\item $C^{k+1}$ vector fields $Y_{g'}$ close to $Y_g$,
	\item $C^{\nu_g}$ maps $\Psi$ from $SM$ to itself, close to identity, and $C^{1+\nu_g}$ along the flow of $Y_g$,
	\item $C^{\nu_g}$ functions $a$ close to $1$,
\end{itemize}
where $\nu_g$ has to be determined. With this topology, $\Xi$ is a $C^k$ map\footnote{The main point here is that the composition as a map $C^{k+1}\times C^\nu\to C^\nu$ is $C^k$ for $\nu\in (0,1)$, instead of $C^{k+1}$ for $\nu=0$. This is discussed at length in \cite{DeLaLLave-Obaya-98}.}, and the idea is to apply the Implicit Function Theorem (IFT). Indeed, if we have a solution of $\Xi(Y_{g'}, \Psi, a) = 0$, then we have 
\begin{equation}\label{eq:conjugation}
\dd \Psi(Y_g) = a \times ( Y_{g'} \circ \Psi), 
\end{equation}
which is exactly an orbit conjugation formula. However, solutions $(\Psi,a)$ are not unique but come in families. Indeed, if $u\in C^{\nu_g}$ and $Y_g u \in C^{\nu_g}$, setting 
\[
\Upsilon_{g,u}(x) := \varphi^{Y_g}_{u(x)}(x),
\]
we find:
\[
\dd \Upsilon_{g,u} (Y_g) = (1+ Y_g u) Y_g\circ \Upsilon_{g,u}. 
\]
Considering $\Psi \circ \Upsilon_{g,u}$, we then find that
\[
\begin{split}
\dd [\Psi \circ \Upsilon_{g,u}](Y_g) & =(1+ Y_g u) \dd_{\Upsilon_{g,u}(x)} \Psi ( Y_g \circ \Upsilon_{g,u}) \\
& = \Big[(1+ Y_g u) \times a \circ \Upsilon_{g,u}\Big] \times ( Y_{g'} \circ \Psi\circ \Upsilon_{g,u}),
\end{split}
\]
that is, the pair $(\Psi',a')$ defined by $\Psi' := \Psi \circ \Upsilon_{g,u}$ and $a' := (1+ Y_g u) \times a \circ \Upsilon_{g,u}$ also satisfies $\Xi(Y_{g'},\Psi',a') = 0$. Conversely, any pair $(\Psi',a')$ which is close to $(\mathrm{id},\mathbf{1})$ is of the form previously described, for some function $u$. Also observe that the previous discussion shows that given $a, u \in C^{\nu_g}(SM)$ such that $Y_{g_0} u\in C^{\nu_g}$, one has:
\begin{equation}
\label{equation:structure}
a = (1+Y_{g_0}u) a \circ \Upsilon_{g_0,u} \mod D^{\nu_g}
\end{equation}
This remark will be used later.

To solve the ambiguity in the definition of the orbit-conjugacy, one has to pick a gauge condition. For example, we can parameterize the Hölder maps close to identity as maps of the form
\[
x \mapsto \exp_x ( V(x)), 
\]
for some $V:SM \to T(SM)$, a $C^{\nu_g}$ vector field (the exponential map here is taken with respect to an arbitrary smooth fixed metric on $SM$). A practical gauge condition is given by 
\[
\text{ for } (x,v) \in SM,\quad V(x,v) \in \ker \beta_g = E^u(g) \oplus E^s(g). 
\]
With this condition, the Implicit Function Theorem applies to $\Xi$, and we find $C^k$ maps
\[
\begin{split}
C^{k+2}(M,\otimes^2_S T^*M) \owns g' &\mapsto Y_{g'} \in C^{k+1}(SM,T(SM)) \\
&\hspace{30pt}  \mapsto ( \Psi_{g\to g'}, a_{g\to g'} ) \in \mathrm{Homeo}^{\nu_g}(SM) \times C^{\nu_g}(SM). 
\end{split}
\]
Here, $\mathrm{Homeo}^{\nu_g}(SM)$ denotes homeomorphisms that are Hölder-continuous.
In order to apply the IFT, we only need to check that the differential $\mathbf{D}_{\Psi,a}\Xi$ is invertible when $Y_{g'}=Y_g$, $\Psi=id$ and $a=\mathbf{1}$. We find that
\[
\mathbf{D}_{\Psi,a}\Xi|_{g'=g,\Psi=1,a=1} (\mathfrak{V},\mathfrak{a}) = \mathcal{L}_{Y_g} \mathfrak{V}  -  \mathfrak{a}Y_g. 
\]
Since $\mathcal{L}_{Y_g} \beta_g = 0$, $\ker \beta_g$ is preserved by $\mathcal{L}_{Y_g}$, and we find:

\begin{lemma}
\label{lemma:dllm}
The differential $\mathbf{D}_{\Psi,a}\Xi|_{g'=g,\Psi=1,a=1}$ is invertible. If $W$ is a vector field, we can decompose $W = c Y_g + W^{\perp}$, with $W^\perp \in \ker \beta_g$. Then
\[
\mathbf{D}_{\Psi,a}\Xi|_{g'=g,\Psi=1,a=1}^{-1}(W) = (\mathfrak{V},\mathfrak{a}),
\]
where $\mathfrak{a}= - c$, and decomposing $W^\perp = W^u + W^s$ along $E^u(g)\oplus E^s(g)$, 
\[
\mathfrak{V} = - \int_0^{+\infty} (\varphi^{Y_g}_t)^\ast W^u dt + \int_{-\infty}^0 (\varphi^{Y_g}_t)^\ast W^s dt := \mathbf{R}_g W^{\perp}. 
\]
There exists $\alpha\in(0,1)$ so that the operator $\mathbf{R}_g : C^s(SM,\ker \beta_g) \mapsto C^s(SM,\ker \beta_g)$ is continuous for $s\in [0,\alpha)$. We have a lower bound 
\[
\alpha \geq \min\left(  \frac{\lambda_s^{\min}}{\lambda_s^{\max}}, \frac{\lambda_u^{\min}}{\lambda_u^{\max}} \right).
\]
\end{lemma}

This lower bound for $\alpha$ can be extracted from the arguments page 593 in \cite{Katok-Knieper-Pollicott-Weiss-89}, and replacing $\phi$ by $\phi^m$ with $m$ large. This proves that $C^2\ni g \mapsto \alpha$ is locally uniformly positive. We will take $\nu_g = \alpha/2$ for example. For our purposes, we will need to understand how the vector field $Y_{g'}$ varies with $g'$:

\begin{lemma}
\label{lemma:diff}
We have: 
\[
\mathbf{D}_{g'} Y_{g'}|_{g'=g} (h) = -\frac{1}{2} (\pi_{2,g}^\ast h) Y_g + W_g^\perp(h),
\]
for some $W_g^\perp(h) \in \ker \beta_g$. Furthermore, $\|(\pi_{2,g}^\ast h)\|_{C^\alpha} \lesssim \|h\|_{C^\alpha}$ and $\|W_g^\perp(h)\|_{C^\alpha} \lesssim \|h\|_{C^{1+\alpha}}$ for all $\alpha \geq 0$.
\end{lemma}

We will limit the regularity $\alpha \leq 2$, because we will use a Taylor-expansion for $C^3$-metrics, thus $h$ will be $C^3$. 

\begin{proof}
We start with the identities $\imath_{Y_g} \beta_g = 1, \imath_{Y_g} d \beta_g = 0$. Differentiating with respect to $g$ in the direction $h \in C^3(M,\otimes^2_S T^*M)$, and writing $\dot{Y}_g := \mathbf{D}_{g'} Y_{g'}|_{g'=g} (h)$, we get:
\[
\imath_{\dot{Y}_g} \beta_g + \imath_{Y_g} \dot{\beta}_g = 0, ~~~ \imath_{\dot{Y}_g} d \beta_g + \imath_{Y_g} d \dot{\beta}_g = 0.
\]
Decomposing $\dot{Y}_g = c_g(h)Y_g + W_g^\perp(h)$, where $c_g(h)$ is a function on $SM$ and $W_g^\perp(h) \in \ker \beta_g$. We have, using \eqref{equation:bam}, for $(x,v) \in SM$:
\[
\begin{split}
c_g(h) = \beta_g(\dot{Y}_g) = - \dot{\beta}_g(Y_g) & =- \left[ -\dfrac{h_x(v,v)}{2g_x(v,v)^{3/2}} g_x(v, \dd \pi_{(x,v)}(Y_g)) + \dfrac{1}{|v|_g} h_x(v, \dd \pi_{(x,v)}(Y_g)) \right]\\
& = - \left[ -\dfrac{h_x(v,v)}{2g_x(v,v)^{3/2}} g_x\left(v, \frac{v}{|v|_g}\right) + \dfrac{1}{|v|_g} h_x\left(v, \frac{v}{|v|_g}\right) \right]\\
& = -1/2 \times h_x(v,v) /|v|_g^2 = -1/2 \times \pi_{2,g}^\ast h(x,v). 
\end{split}
\]

It remains to evaluate the H\"older norm of $W^\perp_g(h)$. By definition, 
\[
\imath_{W_g^\perp(h)} d \beta_g = \imath_{\dot{Y}_g} d \beta_g =-  \imath_{Y_g} d \dot{\beta}_g
\]
characterizes entirely $W_g(h)$ since $d\beta_g$ is a non-degenerate $2$-forms on $\ker \beta_g$. We conclude that $h\mapsto W_g^\perp(h)$ is a linear differential operator of order $1$ with $C^2$ coefficients for $g\in C^3$. The result follows.
\end{proof}

Putting together the lemmata, we get
\begin{lemma}
\label{lemma:derivative-0}
For $h\in C^3(M,\otimes^2_S T^\ast M)$, 
\begin{equation}\label{eq:derivative-at-g-conjugation}
\mathbf{D}_{g'} a_{g\to g'}|_{g'=g}(h) = \frac{1}{2}\pi_{2,g}^\ast h,\quad \mathbf{D}_{g'} \Psi_{g\to g'}|_{g'=g}(h) = \mathbf{R}_g W_g^\perp(h). 
\end{equation}
\end{lemma}

Let us now explain how one can improve the regularity of the stretch when modding out the co-boundaries. For this, it will be more convenient to work with two metrics: $g_0$ is fixed, $g$ is $C^3$ close to $g_0$, and $g'$ is $C^3$ close to $g$. As we have seen the value of $\nu_g$ is locally uniformly positive, so that we can find $\nu_0\in (0,1)$ satisfying $\nu_0 \leq \nu_g$ for all $g$ sufficiently $C^3$ close to $g_0$. The actual value of $\nu_0$ will not be crucial in our argument and can be taken arbitrarily small.

We then have $C^{\nu_0}$ maps $a_{g_0\to g}$, $a_{g\to g'}$ and $\Psi_{g_0\to g}$, $\Psi_{g\to g'}$ on $SM$, depending $C^1$, respectively on $g,g'\in C^3$. They satisfy the relations:
\begin{equation}
\begin{split}
d \Psi_{g_0\to g}( Y_{g_0} ) 	&= a_{g_0 \to g}\times  Y_g \circ \Psi_{g_0\to g},\label{eq:conjugation-bis}\\
 d \Psi_{g\to g'} (Y_g) 		&= a_{g\to g'} \times Y_{g'}\circ \Psi_{g_0\to g}.
\end{split}
\end{equation}
Hence, the pairs $(\Psi_{g_0 \to g'},a_{g_0 \to g'})$ and $(\Psi_{g \to g'} \circ \Psi_{g_0 \to g}, a_{g \to g'} \circ \Psi_{g_0 \to g} a_{g_0 \to g})$ produce orbit-conjugacies for the flows generated by $Y_{g_0}$ and $Y_{g'}$ so there exists $u \in C^{\nu_0}$ such that $Y_{g_0} u \in C^{\nu_0}$ and
\begin{equation}
\label{equation:life}
\begin{split}
& \Psi_{g_0\to g'} = \Psi_{g\to g'} \circ \Psi_{g_0\to g}\circ \Upsilon_{g_0,u},  \\
&a_{g_0 \to g'} = (1 + Y_{g_0} u) \left[a_{g \to g'} \circ \Psi_{g_0 \to g} \times a_{g_0 \to g}\right] \circ \Upsilon_{g_0,u}.
\end{split}
\end{equation}
In particular, $a_{g \to g'} \circ \Psi_{g_0 \to g} a_{g_0 \to g} = a_{g_0 \to g'}$ in the space $C^{\nu_0}/D^{\nu_0}$. We will prove the following:

\begin{lemma}
\label{lemma:derivative-stretch-mod-Dnu}
For $\|g - g_0\|_{C^3} < \eps_0$ small enough, the map
\[
C^3(M,\otimes^2_S T^*M) \ni g \mapsto [a_{g_0 \to g}] \in C^{\nu_0}/D^{\nu_0}(SM),
\]
is $C^2$. Moreover, we have:
\begin{equation}
\label{eq:log-derivative-stretch}
\mathbf{D}_g a_{g_0\to g} (h)  = \frac{1}{2}\pi_{2,g}^\ast h \circ \Psi_{g_0 \to g} \times  a_{g_0 \to g}  \mod  D^{\nu_0},
\end{equation} 
and:
\begin{equation}
\label{equation:seconde-derivative}
\begin{split}
\mathbf{D}^2_{g} a_{g_0 \to g}(h,h) & = \left(-\frac{1}{4} (\pi_{2,g}^\ast h)^2  + \frac{1}{2} \dd \pi_{2,g}^\ast h\left( \mathbf{R}_g W^\bot_g(h) \right)\right) \circ \Psi_{g_0\to g}  \times a_{g_0 \to g} \mod D^{\nu_0}.
\end{split}
\end{equation}
\end{lemma}

\begin{proof}
The computation for the first derivative is an immediate consequence of the previous discussion combined with Lemma \ref{lemma:derivative-0}. From the very expression of $[\mathbf{D}_g a_{g_0\to g} (h)] \in C^{\nu_0}/D^{\nu_0}$, we see that it depends on a $C^1$ fashion of the metric $g \in C^3(M,\otimes^2_S T^*M)$ (since $\Psi_{g_0 \to g}$ and $a_{g_0 \to g}$ depend in a $C^1$ fashion of $g$). Hence, the stretch is $C^2$ as claimed. As far as the second derivative is concerned, we start with the equality in $C^{\nu_0}/D^{\nu_0}$:
\[
\mathbf{D}_{g'} a_{g_0\to g'} (h)  = \frac{1}{2}\pi_{2,g'}^\ast h \circ \Psi_{g_0 \to g'} \times  a_{g_0 \to g'}  \mod  D^{\nu_0}.
\]
Then, using \eqref{equation:life} in the first line, together with \eqref{equation:structure} in the second line, we get:
\[
\begin{split}
\mathbf{D}_{g'} a_{g_0\to g'} (h)  &= \frac{1}{2}\pi_{2,g'}^\ast h \circ \Psi_{g\to g'} \circ \Psi_{g_0\to g}\circ \Upsilon_{g_0,u} \times   (1 + Y_{g_0} u) \left[a_{g \to g'} \circ \Psi_{g_0 \to g} \times a_{g_0 \to g}\right] \circ \Upsilon_{g_0,u}  \\
& \hspace{10cm} \mod  D^{\nu_0} \\
& = \frac{1}{2}\pi_{2,g'}^\ast h \circ \Psi_{g\to g'} \circ \Psi_{g_0\to g} \times \left[a_{g \to g'} \circ \Psi_{g_0 \to g} \times a_{g_0 \to g}\right]  \mod D^{\nu_0}.
\end{split}
\]
Differentiating with respect to $g'$ and evaluating at $g$, we get by Lemma \ref{lemma:derivative-0}:
\[
\begin{split}
\mathbf{D}^2_{g} a_{g_0 \to g}(h,h) & = \frac{1}{2} \mathbf{D}_{g'}(\pi_{2,g'}^\ast h)|_{g'=g} \circ \Psi_{g_0\to g} \times a_{g_0 \to g} + \frac{1}{2} \dd \pi_{2,g}^\ast h\left( \mathbf{R}_g W^\bot_g(h) \circ \Psi_{g_0 \to g} \right) \times a_{g_0 \to g} \\
& + \frac{1}{4} (\pi_{2,g}^\ast h)^2 \circ \Psi_{g_0 \to g}  \times a_{g_0 \to g} \mod D^{\nu_0}.
\end{split}
\]
An easy computation yields $\mathbf{D}_{g'}(\pi_{2,g'}^\ast h)|_{g'=g} = -(\pi_{2,g}^\ast h)^2$, providing the announced result.
\end{proof}

We need to estimate the second derivative of the stretch. Even if the stretch is a $C^2$ map when taking values in $C^{\nu_0}/D^{\nu_0}$, we will evaluate its smoothness in a less regular space, namely in $C^{\nu}$, for $\nu \ll \nu_0$.

\begin{lemma}
\label{lemma:norme}
For $\| g-g_0\|_{C^3}$ small enough, $0<\nu < \nu_0$, and $\alpha = \nu/\nu_0 \leq \nu_0$, there exists a constant $C > 0$ such that:
\[
\|\mathbf{D}^2_g a_{g_0 \to g} (h,h)\|_{C^{\nu}/D^\nu} \leq C \|h\|_{C^{1+\alpha}}\|h\|_{C^1}
\]
\end{lemma}

\begin{proof}
Modulo $D^{\nu_0}$, the second derivative of the stretch is of the form $F \circ \Psi_{g_0 \to g} \times a_{g_0 \to g}$. Hence, using uniform bounds for $a_{g_0 \to g}$ in $C^{\nu}$, we get:
\[
\|F \circ \Psi_{g_0 \to g} \times a_{g_0 \to g}\|_{C^{\nu}} \lesssim \|F \circ  \Psi_{g_0 \to g}\|_{C^{\nu}} \|a_{g_0 \to g}\|_{C^{\nu}} \lesssim \|F\|_{C^{\alpha}} \|\Psi_{g_0 \to g}\|_{C^{\nu/\alpha}}^{\alpha} \lesssim  \|F\|_{C^{\alpha}}.
\]
Hence, it remains to compute the $C^{\alpha}$-norm of
\[
F = -\frac{1}{4} (\pi_{2,g}^\ast h)^2  + \frac{1}{2} \dd \pi_{2,g}^\ast h\left( \mathbf{R}_g W^\bot_g(h) \right).
\]
We then have using Lemmas \ref{lemma:dllm} and \ref{lemma:diff}, for $0\leq\alpha\leq \nu_0$:
\[
\begin{split}
\|F\|_{C^{\alpha}}& \lesssim \|h\|_{C^{\alpha}}^2 + \|\dd \pi_{2,g}^\ast h\|_{C^{\alpha}} \|\mathbf{R}_g W^\bot_g(h)\|_{C^{0}} +  \|\dd \pi_{2,g}^\ast h\|_{C^{0}} \|\mathbf{R}_g W^\bot_g(h)\|_{C^{\alpha}} \\
&  \lesssim \|h\|_{C^{\alpha}}^2 + \|h\|_{C^{1+\alpha}} \|W^\bot_g(h)\|_{C^{0}} + \|h\|_{C^1} \|W^\bot_g(h)\|_{C^{\alpha}} \\
& \lesssim  \|h\|_{C^{\alpha}}^2 + \|h\|_{C^{1+\alpha}} \|h\|_{C^{1}} + \|h\|_{C^1} \|h\|_{C^{1+\alpha}} \lesssim  \|h\|_{C^{1+\alpha}}\|h\|_{C^1}
\end{split}
\]
\end{proof}

We can now prove Proposition \ref{prop:sharp-estimate-stretch}

\begin{proof}[Proof of Proposition \ref{prop:sharp-estimate-stretch}]
The stretch $C^3(M,\otimes^2_S T^*M) \ni g \mapsto [a_{g_0 \to g}] \in C^{\nu_0}/D^{\nu_0}(SM)$ is a $C^2$ map and we can therefore compute its Taylor expansion at $0$:
\[
a_g = \mathbf{1} + \dfrac{1}{2} \pi_2^*(g-g_0) + \int_0^1 (1-t) \mathbf{D}^2_{g'} a_{g_0 \to g'}|_{g'=tg + (1-t)g_0} (g-g_0,g-g_0) \dd t \mod D^{\nu_0}.
\]
Taking the $C^{\nu}$-norm and applying the previous Lemma \ref{lemma:norme}, we obtain by interpolation that:
\[
\left\|a_g - \Big(\mathbf{1} + \dfrac{1}{2} \pi_2^*(g-g_0)\Big)\right\|_{C^{\nu}/D^{\nu}} \lesssim \|g-g_0\|_{C^{1+\alpha}}\|g-g_0\|_{C^{1}}  \lesssim \|g-g_0\|_{C^{\nu-1}} \|g-g_0\|_{C^{3+\alpha(1-\nu_0)}}
\]
(here, the constraint on $\alpha$ is $\alpha \in (0,\nu_0^2)$ so $\alpha(1-\nu_0)$ can be taken arbitrarily small).
\end{proof}

\subsection{Generalized X-ray transform}

It is now well-known \cite{Faure-Sjostrand-11, Giulietti-Liverani-Pollicott-13, Faure-Tsuji-13, Dyatlov-Zworski-16} that the resolvents of the flow $\RR_{\pm}(\lambda) := (\mp X_{g_0}-\lambda)^{-1} : C^\infty(SM) \rightarrow \mc{D}'(SM)$, initially defined on $\left\{\Re(\lambda) > 0 \right\}$ admit a meromorphic extension to the whole complex plane with poles of finite rank, called the \emph{Pollicott-Ruelle resonances}. In particular, near $\lambda = 0$, we have the following expansion (see \cite[Section 2.3]{Guillarmou-17-1}):
\[
\RR_+(\lambda) = -\RR_0^+ - \dfrac{\Pi_0}{\lambda} + \mc{O}(\lambda),\ \RR_-(\lambda) = -\RR_0^- - \dfrac{\Pi_0}{\lambda} + \mc{O}(\lambda), 
\]
for some operators $\RR_0^{\pm} : C^\infty(SM) \rightarrow \mc{D}'(SM)$ which will be further described in a few lines. Up to a normalizing factor, we have $\Pi_0 = \mathbf{1} \otimes \mathbf{1}$, i.e. it is the orthogonal projection onto constant functions. We can then form the operator $\Pi := \RR_0^+ + \RR_0^-$. This operator satisfies $\Pi X f= X \Pi f= 0$ for all $f \in H^s(SM)$, see \cite[Theorem 2.6]{Guillarmou-17-1}. Moreover, it is nonnegative in the sense that $\langle \Pi f, f \rangle_{L^2} \geq 0$, for all $f \in H^s(SM)$, see \cite[Lemma 4.3]{Gouezel-Lefeuvre-19}. We define:
\[
\Pi_2 := {\pi_2}_*(\Pi_0 + \Pi) \pi_2^*
\]
This operator is called the \emph{generalized X-ray transform}. It is a pseudodifferential operator of order $-1$ \cite[Theorem 3.1]{Guillarmou-17-1} with explicit principal symbol computed in \cite[Theorem 4.4]{Gouezel-Lefeuvre-19}. In particular, it is known to be elliptic and invertible on solenoidal tensors (i.e. on $\ker D^*$). A immediate implication is the following bound:

\begin{lemma}
\label{lemma:bound}
For all $s \in \R$, there exists a constant $C = C(s) > 0$ such that:
\[
\forall f \in C^s_*(M,\otimes^2_S T^*M) \cap \ker D^*, ~~~~ \|f\|_{C^s_*} \leq C \|\Pi_2f\|_{C^{s+1}_*}.
\]
\end{lemma}

Eventually, we will need the following key ingredient:

\begin{lemma}
\label{lemma:bound2}
For all $s > 0$, the operator ${\pi_2}_* \Pi : C^s(SM) \rightarrow C^s(M, \otimes^2_S T^*M)$ is bounded.
\end{lemma}

This is a consequence of Theorem \ref{theorem:source}.

\begin{proof}
It is sufficient to argue on ${\pi_2}_* \RR_0^+$ as ${\pi_2}_* \RR_0^-$ is dealt in the same fashion. Recall that one has the splitting
\[
T(SM) = \R \cdot X \oplus \V \oplus \HH,
\]
where $\V = \ker \dd \pi$ (with $\pi : SM \rightarrow M$ being the projection) is the \emph{vertical subspace}, and $\HH$ is the \emph{horizontal subspace}, see \cite{Paternain-99} for further details. We introduce $\V^*(\V)=0,\HH^*(\HH \oplus \R \cdot X) = 0$. As ${\pi_2}_*$ is a pushforward, it only selects wavefront set in $\V^*$. More precisely, if $f \in \mc{D}'(SM)$ is such that $\WF(f) \cap \V^* = \emptyset$, then ${\pi_2}_* f$ is smooth, see the proof of \cite[Theorem 3.1]{Guillarmou-17-1}. As a consequence, it is sufficient to prove that if $u \in C^s(SM)$, then $\RR_0^+ u$ is microlocally $C^s$ near $\V^*$.

As $\RR_0^+$ is the inverse of $X$ on $\ker \Pi_0$, we consider $u \in C^\infty(SM) \cap \ker \Pi_0$ and set $f := \RR_0^+ u$, thus $Xf = u$. In particular, we know that $f$ is microlocally smooth everywhere except near $E_u^*$ that is $\WF(f) \subset E_u^*$, see \cite[Proposition 3.3]{Dyatlov-Zworski-16}. As a consequence, we can apply the source estimate of Theorem \ref{theorem:source} and we obtain for $s > 0, N > 0$:
\[
\|Af\|_{C^s} \leq C\left(\|BXf\|_{C^{s}} + \|f\|_{C^{-N}} \right),
\]
where $A,B \in \Psi^0(SM)$ are microlocalized near $E_s^*$, that is
\[
\|A \RR_0^+ u\|_{C^s} \leq C\left(\|B u\|_{C^{s}} + \|\RR_0^+ u\|_{C^{-N}} \right).
\]
Note that we already know that $\RR_0^+ : C^s \rightarrow C^{-N}$ is bounded as the following holds by \cite[Theorem 2.6]{Guillarmou-17-1}: $C^s \hookrightarrow H^{s/2} \overset{\RR_0^+}{\rightarrow} H^{-s} \hookrightarrow C^{-N}$, for $N \geq 0$ large enough. In other words:
\[
\|A \RR_0^+ u\|_{C^s} \leq C\|u\|_{C^s}.
\]
We now fix a conic neighborhood $\mc{V}$ of $\V^*$ in $T^*(SM)$. As $E_s^*$ is a source for the Hamiltonian dynamics of $(\Phi_t)_{t \in \R}$, for any $(x,\xi) \in \mc{V}$, there exists $T(x,\xi) \geq 0$ such that $\Phi_{-T(x,\xi)}(x,\xi) \in \Ell(A)$ and there is a uniform bound $\sup_{(x,\xi) \in \mc{V}} T(x,\xi) \leq T < \infty$. As a consequence, by propagation of singularities in Hölder-Zygmund spaces (Proposition \ref{prop:usual-propagation-Cs}), we deduce that for $A_1 \in \Psi^0(SM)$ with wavefront set in the conic neighborhood $\mc{V}$ of $\V^*$, there exists $B_1 \in \Psi^0(SM)$ (with wavefront set disjoint from $E_s^*$ and $E_u^*$) such that:
\[
\|A_1f\|_{C^s} \leq C\left(\|Af\|_{C^s} + \|B_1Xf\|_{C^s} + \|f\|_{C^{-N}} \right)
\]
Hence:
\[
\|A_1 \RR_0^+ u\|_{C^s} \leq C \|u\|_{C^s}.
\]
By the wavefront set properties of ${\pi_2}_*$, this implies that ${\pi_2}_* \RR_0^+ : C^s(SM) \rightarrow C^s(M,\otimes^2_S T^*M)$ is bounded.
\end{proof}

\subsection{Proof of Theorem \ref{theorem:mls}} Before completing the proof of Theorem \ref{theorem:mls}, we need a preliminary lemma:

\begin{lemma}
\label{lemma:reduction}
Let $N \geq 2$. Then, there exists $\eps > 0$ such that the following holds. For any metric $g$ such that $\|g-g_0\|_{C_*^N} < \eps$, there exists a (unique) diffeomorphism isotopic to the identity $\phi$, of regularity $C_*^{N+1}$, such that $D^*(\phi^*g) = 0$. The metric $\phi^*g$ is called the \emph{solenoidal reduction} of $g$.
\end{lemma}

We refer to \cite{Ebin-68, Guillarmou-Lefeuvre-18} for a proof.

\begin{proof}[Proof of Theorem \ref{theorem:mls}]
First of all, we define $g' = \phi^*g$ as the solenoidal reduction of $g$ (with respect to $g_0$), i.e. $D^*(g'-g_0) = 0$, by Lemma \ref{lemma:reduction} applied with $N := 3+\eps$. Observe that $a_g$ and $a_{g'}$ are cohomologous since $g$ and $g'$ have same marked length spectrum. Using Lemmas \ref{lemma:bound} and \ref{lemma:bound2}, with $\nu\in (0,1)$, and the fact that coboundaries are in the kernel of $\Pi + \mathbf{1}\otimes\mathbf{1}$, we get:
\[
\|g'-g_0\|_{C^{\nu-1}}  \lesssim \|\Pi_2(g'-g_0)\|_{C^{\nu}} = \|{\pi_2}_* (\Pi + \mathbf{1}\otimes\mathbf{1})\pi_2^\ast (g'-g_0)\|_{C^\nu} \lesssim \| \pi_2^\ast(g'-g_0) \|_{C^\nu/D^\nu}. 
\]
Next, we use Proposition \ref{prop:sharp-estimate-stretch}: for $\nu>0$ small enough,
\[
\| \pi_2^\ast(g'-g_0) \|_{C^\nu/D^\nu} \lesssim \| a_g- \mathbf{1}\|_{C^\nu/D^\nu} + \|g'-g_0\|_{C^{\nu-1}}\|g'-g_0\|_{C^{3+\eps}}, 
\]
so that
\[
\|g'-g_0\|_{C^{\nu-1}}  ( 1- C \|g'-g_0\|_{C^{3+\eps}} ) \lesssim \| a_g- \mathbf{1}\|_{C^\nu/D^\nu},
\]
for some constant $C > 0$.
%
%
%
%
%
%
%
%
Now, the solenoidal reduction also gives that $\|g'-g_0\|_{C^{3+\eps}} \leq C' \|g-g_0\|_{C^{3+\eps}}$ for some constant $C' > 0$. Hence, assuming that $\|g-g_0\|_{C^{3+ \eps}} < 1/(2CC')$ is small enough, we get:
\[
\|g'-g_0\|_{C^{\nu-1}} \lesssim \|a_g - \mathbf{1}\|_{C^\nu/D^{\nu}}.
\]
This provides the announced result.
\end{proof}

\appendix

\section{A subadditive lemma}

\begin{lemma}
\label{lemma:sub}
Let $g: \M \times \R \ni (x,t) \mapsto g(x,t) \in \R$ be a continuous subadditive family, i.e satisfying for $x\in M$, $t,t'\geq 0$:
\[
g(x, t+ t') \leq g(x,t)+g(\varphi_t(x), t'). 
\]
Assume that $g$ is uniformly Lipschitz along flow orbits in the sense that there exists $C >0$ such that for all $x \in \M, t,s \in \R$: $|g(x,t)-g(\varphi_s x,t)| \leq Cs$. Then:
\[
\lim_{T \rightarrow +\infty} \frac{1}{T} \sup_{x \in \M} g(x,T) = \sup_{x \in \M} \lim_{T \rightarrow +\infty} \frac{1}{T} g(x,T).
\]
\end{lemma}

All these limits exist due to Kingman's subadditive ergodic theorem. The inequality $\geq$ is obvious so it only remains to prove $\leq$. We shall actually prove that
\begin{equation}
\label{equation:red}
\lim_{T \rightarrow +\infty} \frac{1}{T} \sup_{x \in \M} g(x,T) =  \sup_{\mu \in \mc{P}_{\mathrm{inv,erg}}} \lim_{T \rightarrow +\infty} \frac{1}{T}\int_{\M} g(x,T) \dd\mu,
\end{equation}
where $\mc{P}_{\mathrm{inv,erg}}$ is the space of all ergodic invariant probability measures. The right-hand side of \eqref{equation:red} is clearly controlled by:
\[
\sup_{\mu \in \mc{P}_{\mathrm{inv,erg}}} \lim_{T \rightarrow +\infty} \frac{1}{T}\int_{\M} g(x,T) \dd\mu \leq  \sup_{x \in \M} \lim_{T \rightarrow +\infty} \frac{1}{T} g(x,T),
\]
which will eventually prove Lemma \ref{lemma:sub}.

\begin{proof}
It suffices to prove \eqref{equation:red}. The inequality $\geq$ is obvious so it only remains to prove $\leq$. This is an application of Kingman's subadditive ergodic theorem. We start by taking an integer $m \geq 1$. We denote by $x_m$ a point in $M$ where $\sup_{x \in \M} g(x, 2^m)$ is attained. Since $\sup_{x \in \M} g(x,T)$ is a subadditive function of $T$, 
\[
\lim_{T \rightarrow +\infty} \frac{1}{T} \sup_{x \in \M} g(x,T) \leq \frac{1}{2^m} g(x_m, 2^m),
\]
and the right-hand side converges to the left-hand side when $m\to + \infty$. We denote by $\mu_m$ the probability measure obtained by averaging functions over the $[0,2^m]$ orbit of $x_m$, namely $\mu_m(f) = 2^{-m} \int_0^{2^m} f(\varphi_s x_m) \dd s$. Certainly, we can extract a sequence of $(m_n)_{n \in \N}$ such that $\mu_{m_n}$ converges weakly to some measure $\mu$ with is invariant under the flow. When we let $m$ tend to $+\infty$ it will be along this subsequence. 

Our next step is to use the second assumption on $g$, so that for $s\in \R$,
\[
g(x_m, 2^m)  \leq g(\varphi_s(x_m), 2^m)  + C s.
\]
It follows that for $ 0 < n_0 < m$, 
\[
g(x_m, 2^m) \leq \frac{1}{2^{n_0}} \int_0^{2^{n_0}} g(\varphi_s(x_m), 2^m) ds + \frac{C 2^{n_0}}{2}. 
\]
Then, we find
\[
\frac{1}{2^m} g(x_m, 2^m) \leq \frac{1}{2^{n_0}} \int_0^{2^{n_0}} \frac{1}{2^m} g(\varphi_s(x_m), 2^m) ds + \frac{C 2^{n_0-m}}{2}. 
\]
We now take $n_0>0$, $m> n_0$ and decompose
\[
g(\varphi_s(x_m), 2^m) \leq g(\varphi_s(x_m), 2^{n_0}) + \dots + g(\varphi_{s+(2^{m-n_0}-1)2^{n_0}}(x), 2^{n_0}). 
\]
It follows that
\begin{align*}
\frac{1}{2^m} g(x_m, 2^m) &\leq \frac{1}{2^{m}} \int_0^{2^{m}} \frac{1}{2^{n_0}} g(\varphi_s(x_m), 2^{n_0}) \dd s + \frac{C 2^{n_0-m}}{2},\\
					&\leq \int_{\M} \frac{1}{2^{n_0}} g(x, 2^{n_0}) \dd \mu_m(x) + \frac{C 2^{n_0-m}}{2}
\end{align*}
Taking $m\to + \infty$, we deduce that for $n_0>0$, 
\[
\lim_{T \rightarrow +\infty} \frac{1}{T} \sup_{x \in \M} g(x,T) \leq \int_{\M}  \frac{1}{2^{n_0}} g(x, 2^{n_0}) d\mu(x).
\]
According to Kingman's subadditive ergodic theorem, the limit as $n_0 \rightarrow \infty$ converges for $\mu$-almost every $x \in \M$ and we get:
\[
\lim_{T \rightarrow +\infty} \frac{1}{T} \sup_{x \in \M} g(x,T) \leq \int_{\M} \lim_{T \rightarrow +\infty} \frac{1}{T} g(x, T) d\mu(x) \leq \sup_{\mu \in \mc{P}_{\mathrm{inv,erg}}} \lim_{T \rightarrow +\infty} \frac{1}{T}\int_{\M} g(x,T) \dd\mu
\]
\end{proof}

\section{Threshold on Sobolev spaces}

\label{appendix:b}

We briefly compare the threshold condition on Sobolev spaces and on Hölder-Zygmund spaces and explain why the latter have usually a better (namely lower) threshold. For the sake of simplicity, let us assume that $\X = X$ is the vector field acting on $C^\infty(\M)$. In this particular case, we have $\omega_+(X)=\omega_-(X)=0$. If $X$ preserves a smooth measure, then it is well-known by the work of Dyatlov-Zworski \cite[Appendix E]{Dyatlov-Zworski-19} that the $L^2$-threshold is also equal to $0$, namely Theorem \ref{theorem:source} holds \emph{verbatim} with $C^\rho_*$ being replaced by $H^\rho$ for any $\rho > 0$. As a consequence, this does not make any difference.

Nevertheless, in the non-volume-preserving case, there is a notable difference. In the definition of $\omega_{\pm}(\X)$, the appearance of $M(T,x)$ stems from the need to find pointwise bounds for the propagator $e^{t\X}:L^\infty(\M,\E) \to L^\infty(\M,\E)$. If we were working with Sobolev spaces, we would have to work with $L^2$-bounds, and $M(T,x)$ should be replaced by $M(T,x) \mathrm{Jac}_{\varphi_{-T}(x)}(\varphi_T)^{1/2}$, where the Jacobian\footnote{It is defined by the equality:
\[
\int_{\mc{M}} f(\varphi_{-t}(x)) \dd \mu(x) = \int_{\M} f(x) \mathrm{Jac}_x(\varphi_t) \dd \mu(x).
\]} is computed with respect to some smooth (arbitrary) measure $\mu$. 

Hence, in the simple case where $\X = X$, $M(T,x)=1$, in order to evaluate the threshold as in \eqref{equation:threshold}, one would be left with the quantity
\begin{equation}
\label{equation:threshold-l2}
\sup_{x \in \M}  \lim_{T\to + \infty} \frac{1}{T}\log \left(\mathrm{Jac}_{\varphi_{-T}(x)}(\varphi_T)^{1/2} \| \dd_x \varphi_{-T}|_{E^u} \|^\rho\right),
\end{equation}
and finding $\rho > 0$ large enough so that this eventually becomes negative. But we claim the following:

\begin{lemma}
\[
\inf\left\{ \rho > 0 ~\middle|~ \sup_{x \in \M}  \lim_{T\to + \infty} \frac{1}{T}\log \left(\mathrm{Jac}_{\varphi_{-T}(x)}(\varphi_T)^{1/2} \| \dd_x \varphi_{-T}|_{E^u} \|^\rho\right) < 0\right\} > 0.
\]
\end{lemma}

\begin{proof}
It suffices to show that for $\rho=0$, one has
\[
\sup_{x \in \M}  \lim_{T\to + \infty} \frac{1}{T}\log \left(\mathrm{Jac}_{\varphi_{-T}(x)}(\varphi_T)^{1/2} \right) > 0.
\]
Also note that $\mathrm{Jac}_{\varphi_{-T}(x)}(\varphi_T)^{1/2} = \mathrm{Jac}_{x}(\varphi_{-T})^{-1/2}$.
We write
\[
\mathrm{Jac}_x(\varphi_{T}) = \exp\left(\int_0^T \mathrm{div}_\mu(X)(\varphi_s(x)) \dd s \right),
\]
where $\mathrm{div}_\mu(X)$ is the divergence of $X$ with respect to $\mu$. Note that by assumption $\mathrm{div}_\mu(X) \neq 0$ and since $\int_{\M} \mathrm{div}_\mu(X) \dd \mu = 0$, this implies that $\mathrm{div}_\mu(X)$ has both signs. Then, writing $\mc{G}$ for the set of periodic orbits, we have:
\[
\begin{split}
\sup_{x \in \M}  \lim_{T\to + \infty} \frac{1}{T}\log \left(\mathrm{Jac}_{\varphi_{-T}(x)}(\varphi_T)^{1/2} \right)& =  \sup_{x \in \M} \frac{1}{2} \lim_{T\to + \infty} \frac{1}{T} \int_{-T}^{0} \mathrm{div}_\mu(X)(\varphi_s(x)) \dd s \\
&  \geq \sup_{\gamma \in \mc{G}} \frac{1}{2\ell(\gamma)} \int_\gamma \mathrm{div}_\mu(X) \dd \gamma.
\end{split}
\]
We claim that this last quantity is strictly positive. Indeed, assume that contrary, namely for all $\gamma \in \mc{G}$, one has $\int_\gamma \mathrm{div}_\mu(X) \dd \gamma \leq 0$. Then, by \cite{Lopes-Thieullen-05}, one can write $\mathrm{div}_\mu(X) = Xu + h$, where $u,Xu,h \in C^\alpha(\M)$ are Hölder-continuous (for some exponent $\alpha > 0$) and $h \leq 0$. Then $\mathrm{div}_{e^{-u} \mu}(X) = h$. Since 
\[
\int_{\M} \mathrm{div}_{e^{-u }\mu}(X) \dd(e^{-u} \mu) = 0 = \int_{M} h \dd (e^{-u} \mu),
\]
we obtain $h \equiv 0$ and thus $\mathrm{div}_\mu(X) = Xu$ for some Hölder-continuous $u$. Applying Theorem \ref{theorem:regularity}, we get that $u$ is actually smooth. But this implies that $X$ preserves the smooth measure $e^{-u}\mu$.
\end{proof}

As a last remark, in the article \cite{Baladi-Tsujii-08}, Baladi and Tsujii construct spaces with $L^1$-norms along submanifolds close to the unstable foliation, in the case of a diffeomorphism $T$ of a manifold $\M$. They obtain an essential radius for the transfer operators $\mathcal{L}f = g\cdot f\circ T$ on their spaces of the form
\[
\lim_{m\to+\infty} \left( \int_{\M} g^{(m)}(x) \lambda^{p,q,m}(x)  \right)^{1/m}.
\]
where $g^{(m)}(x) = g(T^m x)\dots g(x)$, and $\lambda^{p,q,m}(x)$ is an expression involving the norm of $\dd T^m$ restricted to $E^s$ and $E^u$ (assuming $g$ is $C^\delta$ for some $\delta>0$). Translating this expression in the case of flows where $g^{(m})(x)$ has to be replaced by $e^{\int_0^t V \circ \varphi_s \dd s}$, this seems to indicate that it is possible to obtain a threshold for operators $X+V$ of the form
\[
\omega_{+,BT08}(X+V)=\inf\left\{ \rho \ \middle| \ \lim_{T \rightarrow \infty} \frac{1}{T}\log\int_{\M} e^{\int_0^T V\circ \varphi_{-t} dt} \|{d_x\varphi_{-T}}_{|E^u}\|^\rho \dd \mu < 0 \right\},
\]
where $\mu$ is a smooth measure. Our bound is similar but replacing the integral by the supremum of the integrand, so that $\omega_+(X+V) \geq \omega_{+,BT08}(X+V)$ in general. However, manipulation of the spaces appearing in \cite{Baladi-Tsujii-08} is much less easy than $C^s_\ast$ spaces.

\bibliographystyle{alpha}
\bibliography{Biblio}

\begin{thebibliography}{dlLMM86}

\bibitem[AB]{Adam-Baladi-2021}
Alexander Adam and Viviane Baladi.
\newblock Horocycle averages on closed manifolds and transfer operators.
\newblock {\em Arxiv 1809.04062 v3}.
\newblock in preparation.

\bibitem[Ada19]{Adam-19}
Alexander Adam.
\newblock Horocycle averages on closed manifolds and transfer operators.
\newblock {\em Arxiv 1809.04062 v2}, 2019.

\bibitem[BCG95]{Besson-Courtois-Gallot-95}
G.~Besson, G.~Courtois, and S.~Gallot.
\newblock Entropies et rigidit\'{e}s des espaces localement sym\'{e}triques de
  courbure strictement n\'{e}gative.
\newblock {\em Geom. Funct. Anal.}, 5(5):731--799, 1995.

\bibitem[BFL90]{Benoist-Foulon-Labourie-90}
Yves Benoist, Patrick Foulon, and Fran\c{c}ois Labourie.
\newblock Flots d'anosov \`a distributions de liapounov diff\'erentiables. {I}.
\newblock {\em Annales de l'I.H.P. Physique th\'eorique}, 53(4):395--412, 1990.

\bibitem[BFL92]{Benoist-Foulon-Labourie-92}
Yves Benoist, Patrick Foulon, and Fran\c{c}ois Labourie.
\newblock Flots d'{A}nosov \`a distributions stable et instable
  diff\'{e}rentiables.
\newblock {\em J. Amer. Math. Soc.}, 5(1):33--74, 1992.

\bibitem[BK85]{Burns-Katok-85}
K.~Burns and A.~Katok.
\newblock Manifolds with nonpositive curvature.
\newblock {\em Ergodic Theory Dynam. Systems}, 5(2):307--317, 1985.

\bibitem[BT07]{Baladi-Tsujii-07}
Viviane Baladi and Masato Tsujii.
\newblock Anisotropic {H}\"{o}lder and {S}obolev spaces for hyperbolic
  diffeomorphisms.
\newblock {\em Ann. Inst. Fourier (Grenoble)}, 57(1):127--154, 2007.

\bibitem[BT08]{Baladi-Tsujii-08}
Viviane Baladi and Masato Tsujii.
\newblock Dynamical determinants and spectrum for hyperbolic diffemorphisms.
\newblock {\em Geometric and probabilistic structures in dynamics}, Amer. Math.
  Soc., Providence, RI(469):29--68, 2008.

\bibitem[CLa]{Cekic-Lefeuvre-21-2}
Mihajlo {Cekić} and Thibault {Lefeuvre}.
\newblock {On transparent manifolds}.
\newblock {\em in preparation}.

\bibitem[CLb]{Cekic-Lefeuvre-21-1}
Mihajlo {Cekić} and Thibault {Lefeuvre}.
\newblock {The holonomy inverse problem on Anosov manifolds}.
\newblock {\em in preparation}.

\bibitem[CL21]{Cekic-Lefeuvre-21}
Mihajlo Cekić and Thibault Lefeuvre.
\newblock Generic injectivity of the x-ray transform.
\newblock {\em in preparation}, 2021.

\bibitem[Cro90]{Croke-90}
Christopher~B. Croke.
\newblock Rigidity for surfaces of nonpositive curvature.
\newblock {\em Comment. Math. Helv.}, 65(1):150--169, 1990.

\bibitem[CS98]{Croke-Sharafutdinov-98}
Christopher~B. Croke and Vladimir~A. Sharafutdinov.
\newblock Spectral rigidity of a compact negatively curved manifold.
\newblock {\em Topology}, 37(6):1265--1273, 1998.

\bibitem[DD13]{Datchev-Dyatlov-13}
Kiril Datchev and Semyon Dyatlov.
\newblock Fractal {W}eyl laws for asymptotically hyperbolic manifolds.
\newblock {\em Geom. Funct. Anal.}, 23(4):1145--1206, 2013.

\bibitem[{de }18]{ColinDeVerdiere-18}
Yves~Colin {de Verd{\`\i}{\`e}re}.
\newblock {Spectral theory of pseudo-differential operators of degree 0 and
  application to forced linear waves}.
\newblock {\em arXiv e-prints}, page arXiv:1804.03367, April 2018.

\bibitem[DH72]{Duistermaat-Hormander-72}
J.~J. Duistermaat and L.~H{\"{o}}rmander.
\newblock Fourier integral operators. {II}.
\newblock {\em Acta Math.}, 128(3-4):183--269, 1972.

\bibitem[dlL01]{dlLLave-2001}
Rafael de~la Llave.
\newblock Remarks on {S}obolev regularity in {A}nosov systems.
\newblock {\em Ergodic Theory Dynam. Systems}, 21(4):1139--1180, 2001.

\bibitem[dlLMM86]{DeLaLlave-Marco-Moryon-86}
R.~de~la Llave, J.~M. Marco, and R.~Moriy\'{o}n.
\newblock Canonical perturbation theory of {A}nosov systems and regularity
  results for the {L}iv\v{s}ic cohomology equation.
\newblock {\em Ann. of Math. (2)}, 123(3):537--611, 1986.

\bibitem[DlLO98]{DeLaLLave-Obaya-98}
Rafael De~la Llave and R.~Obaya.
\newblock Regularity of the composition operator in spaces of hölder
  functions.
\newblock {\em Discrete and Continuous Dynamical Systems}, 5, 02 1998.

\bibitem[Dya12]{Dyatlov-12}
Semyon Dyatlov.
\newblock Asymptotic distribution of quasi-normal modes for {K}err--de {S}itter
  black holes.
\newblock {\em Ann. Henri Poincar\'{e}}, 13(5):1101--1166, 2012.

\bibitem[DZ16]{Dyatlov-Zworski-16}
Semyon Dyatlov and Maciej Zworski.
\newblock Dynamical zeta functions for {A}nosov flows via microlocal analysis.
\newblock {\em Ann. Sci. \'{E}c. Norm. Sup\'{e}r. (4)}, 49(3):543--577, 2016.

\bibitem[DZ19a]{Dyatlov-Zworski-19}
Semyon Dyatlov and Maciej Zworski.
\newblock {\em Mathematical theory of scattering resonances}, volume 200 of
  {\em Graduate Studies in Mathematics}.
\newblock American Mathematical Society, Providence, RI, 2019.

\bibitem[DZ19b]{Dyatlov-Zworski-19-2}
Semyon Dyatlov and Maciej Zworski.
\newblock Microlocal analysis of forced waves.
\newblock {\em Pure Appl. Anal.}, 1(3):359--384, 2019.

\bibitem[Ebi68]{Ebin-68}
David~G. Ebin.
\newblock On the space of {R}iemannian metrics.
\newblock {\em Bull. Amer. Math. Soc.}, 74:1001--1003, 1968.

\bibitem[FS11]{Faure-Sjostrand-11}
Fr\'{e}d\'{e}ric Faure and Johannes Sj\"{o}strand.
\newblock Upper bound on the density of {R}uelle resonances for {A}nosov flows.
\newblock {\em Comm. Math. Phys.}, 308(2):325--364, 2011.

\bibitem[FT13]{Faure-Tsuji-13}
Fr\'{e}d\'{e}ric Faure and Masato Tsujii.
\newblock Band structure of the {R}uelle spectrum of contact {A}nosov flows.
\newblock {\em C. R. Math. Acad. Sci. Paris}, 351(9-10):385--391, 2013.

\bibitem[GDP]{Guillarmou-Poyferre-21}
Colin Guillarmou and Thibault De~Poyferré.
\newblock A paradifferential approach for hyperbolic dynamical systems and
  applications.
\newblock {\em In preparation}.

\bibitem[GKL19]{Guillarmou-Knieper-Lefeuvre-19}
Colin {Guillarmou}, Gerhard {Knieper}, and Thibault {Lefeuvre}.
\newblock {Geodesic stretch, pressure metric and marked length spectrum
  rigidity}.
\newblock {\em arXiv e-prints}, page arXiv:1909.08666, Sep 2019.

\bibitem[GL19a]{Gouezel-Lefeuvre-19}
S{\'e}bastien {Gou{\"e}zel} and Thibault {Lefeuvre}.
\newblock {Classical and microlocal analysis of the X-ray transform on Anosov
  manifolds}.
\newblock {\em arXiv e-prints}, Apr 2019.

\bibitem[GL19b]{Guillarmou-Lefeuvre-18}
Colin Guillarmou and Thibault Lefeuvre.
\newblock The marked length spectrum of {A}nosov manifolds.
\newblock {\em Ann. of Math. (2)}, 190(1):321--344, 2019.

\bibitem[GLP13]{Giulietti-Liverani-Pollicott-13}
P.~Giulietti, C.~Liverani, and M.~Pollicott.
\newblock Anosov flows and dynamical zeta functions.
\newblock {\em Ann. of Math. (2)}, 178(2):687--773, 2013.

\bibitem[Gui17]{Guillarmou-17-1}
Colin Guillarmou.
\newblock Invariant distributions and {X}-ray transform for {A}nosov flows.
\newblock {\em J. Differential Geom.}, 105(2):177--208, 2017.

\bibitem[Hal15]{Hall-15}
Brian Hall.
\newblock {\em Lie groups, {L}ie algebras, and representations}, volume 222 of
  {\em Graduate Texts in Mathematics}.
\newblock Springer, Cham, second edition, 2015.
\newblock An elementary introduction.

\bibitem[Ham99]{Hamenstadt-99}
U.~Hamenst\"{a}dt.
\newblock Cocycles, symplectic structures and intersection.
\newblock {\em Geom. Funct. Anal.}, 9(1):90--140, 1999.

\bibitem[Has92]{Hasselblatt-92}
Boris Hasselblatt.
\newblock Bootstrapping regularity of the {A}nosov splitting.
\newblock {\em Proc. Amer. Math. Soc.}, 115(3):817--819, 1992.

\bibitem[HK90]{Hurder-Katok-90}
S.~Hurder and A.~Katok.
\newblock Differentiability, rigidity and {G}odbillon-{V}ey classes for
  {A}nosov flows.
\newblock {\em Inst. Hautes \'{E}tudes Sci. Publ. Math.}, (72):5--61 (1991),
  1990.

\bibitem[HMV04]{Hassell-Melrose-Vasy-04}
Andrew Hassell, Richard Melrose, and Andr\'{a}s Vasy.
\newblock Spectral and scattering theory for symbolic potentials of order zero.
\newblock {\em Adv. Math.}, 181(1):1--87, 2004.

\bibitem[HV18]{Hintz-Vasy-18}
Peter Hintz and Andr\'{a}s Vasy.
\newblock The global non-linear stability of the {K}err--de {S}itter family of
  black holes.
\newblock {\em Acta Math.}, 220(1):1--206, 2018.

\bibitem[Jou86]{Journe-86}
Jean-Lin Journ\'{e}.
\newblock On a regularity problem occurring in connection with {A}nosov
  diffeomorphisms.
\newblock {\em Comm. Math. Phys.}, 106(2):345--351, 1986.

\bibitem[Kat88]{Katok-88}
Anatole Katok.
\newblock Four applications of conformal equivalence to geometry and dynamics.
\newblock {\em Ergodic Theory Dynam. Systems}, 8$^*$(Charles Conley Memorial
  Issue):139--152, 1988.

\bibitem[KH95]{Hasselblatt-Katok-95}
Anatole Katok and Boris Hasselblatt.
\newblock {\em Introduction to the modern theory of dynamical systems},
  volume~54 of {\em Encyclopedia of Mathematics and its Applications}.
\newblock Cambridge University Press, Cambridge, 1995.
\newblock With a supplementary chapter by Katok and Leonardo Mendoza.

\bibitem[KKPW89]{Katok-Knieper-Pollicott-Weiss-89}
A.~Katok, G.~Knieper, M.~Pollicott, and H.~Weiss.
\newblock Differentiability and analyticity of topological entropy for anosov
  and geodesic flows.
\newblock {\em Inventiones mathematicae}, 98:581--597, 1989.

\bibitem[Liv72]{Livsic-72}
A.~N. Liv\v{s}ic.
\newblock Cohomology of dynamical systems.
\newblock {\em Izv. Akad. Nauk SSSR Ser. Mat.}, 36:1296--1320, 1972.

\bibitem[LT05]{Lopes-Thieullen-05}
A.~O. Lopes and Ph. Thieullen.
\newblock Sub-actions for {A}nosov flows.
\newblock {\em Ergodic Theory Dynam. Systems}, 25(2):605--628, 2005.

\bibitem[Mel94]{Melrose-94}
Richard~B. Melrose.
\newblock Spectral and scattering theory for the {L}aplacian on asymptotically
  {E}uclidian spaces.
\newblock In {\em Spectral and scattering theory ({S}anda, 1992)}, volume 161
  of {\em Lecture Notes in Pure and Appl. Math.}, pages 85--130. Dekker, New
  York, 1994.

\bibitem[NT98]{Nitica-Torok-98}
Viorel Ni\c{t}ic\u{a} and Andrei T\"{o}r\"{o}k.
\newblock Regularity of the transfer map for cohomologous cocycles.
\newblock {\em Ergodic Theory Dynam. Systems}, 18(5):1187--1209, 1998.

\bibitem[Ota90]{Otal-90}
Jean-Pierre Otal.
\newblock Le spectre marqu\'{e} des longueurs des surfaces \`a courbure
  n\'{e}gative.
\newblock {\em Ann. of Math. (2)}, 131(1):151--162, 1990.

\bibitem[Pat99]{Paternain-99}
Gabriel~P. Paternain.
\newblock {\em Geodesic flows}, volume 180 of {\em Progress in Mathematics}.
\newblock Birkh\"{a}user Boston, Inc., Boston, MA, 1999.

\bibitem[PSU14]{Paternain-Salo-Uhlmann-14-1}
Gabriel~P. Paternain, Mikko Salo, and Gunther Uhlmann.
\newblock Tensor tomography: progress and challenges.
\newblock {\em Chin. Ann. Math. Ser. B}, 35(3):399--428, 2014.

\bibitem[RS96]{Runst-Sickel-96}
Thomas Runst and Winfried Sickel.
\newblock {\em Sobolev spaces of fractional order, {N}emytskij operators, and
  nonlinear partial differential equations}, volume~3 of {\em De Gruyter Series
  in Nonlinear Analysis and Applications}.
\newblock Walter de Gruyter \& Co., Berlin, 1996.

\bibitem[Vas13]{Vasy-13}
Andr\'{a}s Vasy.
\newblock Microlocal analysis of asymptotically hyperbolic and {K}err-de
  {S}itter spaces (with an appendix by {S}emyon {D}yatlov).
\newblock {\em Invent. Math.}, 194(2):381--513, 2013.

\bibitem[{Wan}20]{Wang-20}
Jian {Wang}.
\newblock {Sharp radial estimates in Besov spaces}.
\newblock {\em arXiv e-prints}, page arXiv:2003.11218, March 2020.

\bibitem[Zwo12]{Zworski-12}
Maciej Zworski.
\newblock {\em Semiclassical analysis}, volume 138 of {\em Graduate Studies in
  Mathematics}.
\newblock American Mathematical Society, Providence, RI, 2012.

\end{thebibliography}

\end{document}